\documentclass[a4,reqno]{amsart}
\usepackage[english]{babel}
\usepackage{xcolor}
\usepackage{enumerate}

\usepackage[colorlinks=false,backref=false,pagebackref=false,pdftex,pdfauthor={D. Bonheure, E. Moreira dos Santos, H. Tavares},pdftitle={}]{hyperref}
\usepackage{amsmath,amsthm,amsfonts,latexsym,amssymb, hyperref}
\usepackage{comment}
\usepackage[showonlyrefs]{mathtools}
\mathtoolsset{showonlyrefs=true}

\newcommand{\IR}{{\mathbb R}}

\newcommand{\R}{{\mathbb R}}

\newcommand{\N}{{\mathbb N}}

\newcommand{\Ncal}{{\mathcal N}}
\newcommand{\Hcal}{{\mathcal H}}

\def\a{\alpha}
\def\b{\beta}

\newcommand {\con}{\subset}

\newcommand{\eps}{\varepsilon}

\begin{document}

\title[Existence and symmetry of least energy nodal solutions]
{Existence and symmetry of least energy nodal solutions for Hamiltonian elliptic systems}

\author[D. Bonheure]{Denis Bonheure}
\address{Denis Bonheure \newline \indent D{\'e}partement de Math{\'e}matique \newline \indent  Universit{\'e} libre de Bruxelles \newline \indent CP 214,  Boulevard du Triomphe, B-1050 Bruxelles, Belgium}
\email{denis.bonheure@ulb.ac.be}

\author[E. Moreira dos Santos]{Ederson Moreira dos Santos}
\address{Ederson Moreira dos Santos \newline \indent Instituto de Ci{\^e}ncias Matem{\'a}ticas e de Computa\c{c}{\~a}o - ICMC \newline \indent Universidade de S{\~a}o Paulo - USP \newline \indent
Caixa Postal 668, CEP 13560-970 - S\~ao Carlos - SP - Brazil}
\email{ederson@icmc.usp.br}
\author[M. Ramos]{Miguel Ramos}
\address{Miguel Ramos passed away in January 2013}

\author[H. Tavares]{Hugo Tavares}\thanks{}
\address{Hugo Tavares \newline \indent Center for Mathematical Analysis, Geometry and Dynamical Systems,
 \newline \indent Mathematics Department, \newline \indent Instituto Superior T\'ecnico, Universidade de Lisboa   \newline \indent
Av. Rovisco Pais, 1049-001 Lisboa, Portugal}
\email{htavares@math.ist.utl.pt}

\date{\today}



\begin{abstract}
In this paper we prove existence of least energy nodal solutions for the Hamiltonian elliptic system with H\'enon--type weights 
\[
-\Delta u = |x|^{\beta} |v|^{q-1}v, \quad -\Delta v =|x|^{\alpha}|u|^{p-1}u\quad \mbox{ in } \Omega, \qquad u=v=0 \mbox{ on } \partial \Omega,
\]
where $\Omega$ is a bounded smooth  domain in $\R^N$, $N\geq 1$, $\alpha, \beta \geq 0$ and the nonlinearities are superlinear and subcritical, namely 
\[
1> \frac{1}{p+1}+\frac{1}{q+1}> \frac{N-2}{N}.
\]
When $\Omega$ is either a ball or an annulus centered at the origin and $N \geq 2$, we show that these solutions display the so-called foliated Schwarz symmetry. It is natural to conjecture that these solutions are not radially symmetric. We provide such a symmetry breaking in a range of parameters where the solutions of the system behave like the solutions of a single equation. Our results on the above system are new even in the case of the Lane-Emden system (i.e.\ without weights). As far as we know, this is the first paper that contains results about least energy nodal solutions for strongly coupled elliptic systems and their symmetry properties.

\end{abstract}
\maketitle
\numberwithin{equation}{section}
\newtheorem{theorem}{Theorem}[section]
\newtheorem{lemma}[theorem]{Lemma}
\newtheorem{example}[theorem]{Example}
\newtheorem{remark}[theorem]{Remark}
\newtheorem{proposition}[theorem]{Proposition}
\newtheorem{definition}[theorem]{Definition}
\newtheorem{corollary}[theorem]{Corollary}
\newtheorem*{open}{Open problem}

\section{Introduction}

We consider the Hamiltonian elliptic system with H\'enon-type weights
\begin{equation}\label{eq:mainsystem}
-\Delta u = |x|^{\beta} |v|^{q-1}v, \qquad -\Delta v =|x|^{\alpha}|u|^{p-1}u\quad \mbox{ in } \Omega, \qquad u=v=0 \mbox{ on } \partial \Omega,
\end{equation}
where $\Omega$ is a bounded domain in $\IR^N$, $N\geqslant 1$, and $\alpha, \beta\geqslant  0$. We consider superlinear and subcritical nonlinearities, namely
\begin{equation}\label{eq:subcritical}\tag{H}
1> \frac{1}{p+1}+\frac{1}{q+1}>\frac{N-2}{N}.
\end{equation}
Observe that the first condition is also equivalent to $pq>1$. 

The system \eqref{eq:mainsystem} is strongly coupled in the sense that $u\equiv 0$ if and only if $v\equiv 0$. Moreover, $u$ changes sign if and only if $v$ changes sign.

We recall that a strong solution to this problem corresponds to a pair $(u,v)$ with 
\[
u\in W^{2,(q+1)/q}(\Omega)\cap W^{1,(q+1)/q}_0(\Omega),\qquad v\in W^{2,(p+1)/p}(\Omega)\cap W^{1,(p+1)/p}_0(\Omega)
\]
satisfying the system in \eqref{eq:mainsystem} for a.e. $x\in \Omega$. By using a bootstrap method (see \cite[Theorem 1(a)]{Sirakov}), it can be shown that strong solutions are actually classical solutions.

Consider the energy functional
\begin{equation}\label{energy E}
E(u,v)= \int_\Omega \nabla u\cdot \nabla v\, dx-\frac{1}{p+1}\int_\Omega |x|^{\alpha}|u|^{p+1}\, dx-\frac{1}{q+1}\int_\Omega |x|^\beta |v|^{q+1}\, dx,
\end{equation}
which is well defined for strong solutions thanks to assumption \eqref{eq:subcritical}.

One can use various variational settings to deal with the system \eqref{eq:mainsystem}, see for instance the surveys \cite{BonheureSantosTavares, deFigueiredo,Ruf}. Once the existence of at least one critical point is proved, a natural question is that of the existence of a least energy one, by which we mean a critical point at the level
\[
c=\inf\{E(u,v):\ (u,v) \text{ is a nonzero strong solution of } \eqref{eq:mainsystem}\}.
\]
The solutions at this energy are usually referred to as ground state solutions and in many problems, they are of special interest. In our setting the existence of such solutions is clear and rely on a simple compactness argument. On the other hand, it is useful to get a variational characterisation of these solutions to derive qualitative properties, see for example \cite{BonheureSantosRamosTAMS,BonheureSantosRamosJFA,BonheureSantosTavares}. In our setting, any solution at level $c$ is positive (or negative). This can be established using a Nehari type characterization of the level $c$. We emphasize that the adequate associated Nehari manifold is then of infinite codimension. We refer to \cite{BonheureSantosTavares} for more details.

\medbreak

Existence of sign-changing solutions has been obtained in \cite[Theorem 4]{RamosTavaresZou}, under the extra assumption $p>1$ and $q>1$, where it is proved that there exists an unbounded sequence of solutions $(u_k , v_k )$ such that both $(u_k + v_k )^+\ne 0$ and $(u_k + v_k )^-\ne 0$ for every $k$. In fact, for a pair of solutions $(u,v)$, $u+v$ changes sign if and only if $u$ and $v$ change sign. Our results therefore present some improvement of the result given in \cite[Theorem 4]{RamosTavaresZou} since, by imposing the mere super linearity condition $pq>1$, we are able to prove the existence of a (least energy) nodal solution to \eqref{eq:mainsystem}.

\medbreak

Define then the \emph{least energy nodal level} as
\[
c_{\rm nod}=\inf\{E(u,v):\ (u,v) \text{ is a nonzero strong solution of } \eqref{eq:mainsystem} \text{ and }u^\pm,v^\pm\not \equiv 0\}.
\]
It is not obvious that this level is achieved since this no more follows from a simple minimisation argument. Indeed, even if we have enough compactness to extract a converging subsequence, the limit could be a critical point $(u,v)$ such that both $u$ and $v$ are positive (or negative). The existence of a least energy nodal solution for the scalar Lane-Emden equation \cite{CastroCossioNeuberger,BartschWethTMNA,BartschWethAIHP} follows from the minimisation of the functional over a nodal Nehari set. 
It is not clear at all how such a nodal Nehari set associated to the energy functional $E$ could be defined. Anyhow, our first main result shows $c_{\rm nod}$ is achieved.

\begin{theorem}\label{thm:main1}
Let $N\geq 1$, $\alpha \geq 0$, $\beta\geq 0$ and suppose that \eqref{eq:subcritical} is satisfied. Then the level $c_\text{\rm nod}$ is achieved, that is, there exists a strong solution $(u,v)$ of \eqref{eq:mainsystem} such that $u^{\pm},v^{\pm}\not\equiv 0$ and $E(u,v)=c_{\rm nod}$.
\end{theorem}

Our proof relies on a dual method as in Cl\'ement and van der Vorst \cite{ClementvanderVorst} or Alves and Soares \cite{AlvesSoares}\footnote{In order to apply the dual variational method, the two potentials on the left hand sides of \cite[eq. (1.3)]{AlvesSoares} must be equal as follows from \cite[line 5 p. 114]{AlvesSoares}.} who deal with the singularly perturbed system
\begin{equation}\label{system claudianor}
- \varepsilon^2\Delta u+V(x)u=|v|^{q-1}v,\qquad - \varepsilon^2\Delta v+V(x)v=|u|^{p-1}u\qquad \text{ in } \R^N
\end{equation}
assuming the extra assumptions $p> 1$ and $q>1$. With respect to \cite{AlvesSoares}, the hypothesis \eqref{eq:subcritical} includes more general powers, namely $pq>1$ is enough. This means in particular that we cover the biharmonic operator with H\'enon weight, that is $q=1$ and $\beta=0$, with Navier boundary conditions. In this context, the problem \eqref{eq:mainsystem} reads as
\begin{equation}\label{eq:quarta_ordem}
\Delta^2 u= |x|^{\alpha}|u|^{p-1}u \quad \text{ in } \Omega, \qquad u=\Delta u=0\quad \text{ on } \partial \Omega,
\end{equation}
with $\alpha\geq0$, $\frac{1}{2}> \frac{1}{p+1}> \frac{N-4}{2N}$ and Theorem \ref{thm:main1} applies. 

\medbreak

Next we investigate the symmetry of these solutions in case the domain is radial. Let $\Omega$ be either a ball or an annulus centred at the origin. Recall that a function $u:\Omega\to \R$ is called foliated Schwarz symmetric with
respect to some unit vector $p\in \R^N$ if, for a.e. $r > 0$ such that $\partial B_r(0)\subset \Omega$ and for every $c\in \R$, the restricted superlevel set $\{x \in \partial B_r(0) : u(x) \geq c\}$ is either equal to $\partial B_r(0)$ or to a geodesic ball in $\partial B_r(0)$ centred at $rp$. In other words, $u$ is foliated Schwarz symmetric if $u$ is axially symmetric with respect to the axis $\R p$ and non increasing in the polar angle $\theta= \arccos( x\cdot  p) \in [0, \pi]$.  

In the past thirty years the study on the symmetry properties of positive or least energy solutions of strongly coupled elliptic systems has been an active research field, see for instance \cite{Troy,Shaker,deFigueiredo1994,deFigueiredoYang,Zou,BuscaSirakov,QuittnerSouplet,BonheureSantosRamosTAMS,BonheureSantosRamosJFA,BonheureSantosTavares,DamascelliPacella,DamascelliGladialiPacella2, DamasceliGladialiPacella}. The basic tools used to derive the symmetry of the solutions are the method of moving planes \cite{GidasNiNirenberg}, symmetrization or polarization and Morse index arguments. All these techniques were originally developed for second order elliptic equations and their use in the context of strongly coupled elliptic system requires more involved arguments. 

\medbreak

Our second main result is the following.
\begin{theorem}\label{thm:main2}
Let $\Omega \con \R^N$, $N \geq 2$, be either a ball or an annulus centred at the origin. Let $(u,v)$ be a least energy nodal solution of \eqref{eq:mainsystem}. Then there exists $p\in \partial B_1(0)$ such that both $u$ and $v$ are foliated Schwarz symmetric with respect to $p$.
\end{theorem}

\renewcommand{\thefootnote}{\fnsymbol{footnote}}

As mentioned above, our results cover the biharmonic operator complementing therefore some of the results in \cite{WethTMNA2006}. 

\begin{corollary}\label{th:biharmonic}
Let $\Omega \con \R^N$, $N \geq 1$ and assume that $\frac{1}{2}> \frac{1}{p+1}>\frac{N-4}{2N}$. Then the fourth order problem \eqref{eq:quarta_ordem} admits a least energy nodal solution. Moreover, if $\Omega$  is either a ball or an annulus centred at the origin, $N\geq 2$, then any least energy nodal solution of \eqref{eq:quarta_ordem} is such that $u$ and $-\Delta u$ are foliated Schwarz symmetric with respect to the the same unit vector $p\in \R^N$.
\end{corollary}

For the scalar Lane-Emden equation (i.e.\ without weights), it is known that any least energy nodal solution has Morse index $2$. Combined with the analysis of the Morse index of the sign changing radial solutions when $\Omega$ is either a ball or an annulus, this leads to the conclusion that whereas least energy solutions are radially symmetric, least energy nodal solutions are not. The foliated Schwarz symmetry is thus somehow optimal. 

For the H\'enon-Lane-Emden system \eqref{eq:mainsystem}, it is not clear how to compute (or even define) the Morse index of the solutions. Although we conjecture that for any $p,q$ satisfying \eqref{eq:subcritical} and $\alpha,\beta\geq 0$ every least energy nodal solutions of \eqref{eq:mainsystem} are non radial, we are not able to prove it. Symmetry breaking occurs at least for $p\sim q$, $\alpha \sim 0$ and $\beta\sim 0$. 

\begin{theorem}\label{thm:symmetrybreak_close_diag}
Assume $N \geq 2$ and $\Omega\subset \R^N$  is either a ball or an annulus centred at the origin. Let $q_0$ satisfy
\begin{equation}\label{assumption on q_0 intro}
q_0>1\quad \text{ and } \quad q_0+1<2N/(N-2)\text{ if } N\geq 3.
\end{equation}
Then there exists $\delta_0>0$ such that, if $p,q\in [q_0-\delta_0,q_0+\delta_0]$, $\alpha,\beta\in [0,\delta_0]$, then any least energy nodal solution $(u,v)$ of \eqref{eq:mainsystem} is such that both $u$ and $v$ are non radially symmetric.
\end{theorem}

At this point, we emphasize that when $\Omega\subset \R^N$ is either a ball or an annulus, we can work in a functional framework of radially symmetric functions yielding at least one radial sign-changing solution having least energy among all radial nodal solutions. The previous theorem gives therefore a range of coexistence of both radially symmetric and non radially symmetric sign-changing  solutions. When $p>1$ and $q>1$, the existence of infinitely many radial sign-changing solutions follows also from applying the method of \cite[Theorem 4]{RamosTavaresZou} in a functional framework of radially symmetric functions. Although these solutions can a priori coincide with the solutions obtained in \cite[Theorem 4]{RamosTavaresZou}, we clearly do not expect that to happen. A difficult question seems to be that of providing a precise information on the number of nodes of radial nodal solutions. Both the gluing approach \cite{BartschWillem} and an optimal partition method \cite{ContiTerraciniVerzini}, reminiscent of the original approach of Nehari, seem out of reach for the Lane-Emden system (without weight). Therefore, the existence of radially symmetric solutions with a prescribed number of nodes is a challenging open question.

Finally, we show that all of our results apply to the H\'enon equation
\begin{equation}\label{eq:henon}
-\Delta u = |x|^{\alpha}|u|^{p-1}u \quad \mbox{ in } \Omega, \qquad u =0 \mbox{ on } \partial \Omega.
\end{equation}
For this purpose we prove a kind of symmetry theorem for the components, which guarantees that when $p=q$, $\alpha=\beta$, then any solution $(u,v)$ of \eqref{eq:mainsystem} is such that $u=v$ and whence $u$ solves \eqref{eq:henon}. For related classifications results we mention \cite[Theorem 1.2]{QuittnerSouplet} and \cite{Farina}.

\begin{theorem}\label{lemma:u=v}
Assume that $u,v\in H^1_0(\Omega)$ solves the system \eqref{eq:mainsystem}
\begin{equation}\label{eq:system-p=q}
-\Delta u = |x|^\b |v|^{q-1}v, \qquad -\Delta v =|x|^\b |u|^{q-1}u\quad \mbox{ in } \Omega, \qquad u=v=0 \mbox{ on } \partial \Omega,
\end{equation}
where $\Omega$ is a bounded domain in $\IR^N$, $N\geqslant 1$, $q>1$ and $q+1\le 2N/(N-2)$ if $N\ge 3$. Then $u=v$.
\end{theorem}

Combining Theorems \ref{thm:main1}, \ref{thm:main2},  \ref{thm:symmetrybreak_close_diag} and \ref{lemma:u=v}, we get the following results about least energy nodal solutions of the H\'enon equation \eqref{eq:henon}. 

\begin{corollary}\label{th:eq.henon}
Let $N\geq 1$, $\alpha \geq 0$, $p>1$ and $p+1< 2N/(N-2)$ if $N\geq 3$. Then:
\begin{enumerate}[i)]
\item There exist least energy nodal solutions of \eqref{eq:henon}.
\item Let $\Omega \con \R^N$, $N \geq 2$, be either a ball or an annulus centred at the origin. Let $u$ be a least energy nodal solution of \eqref{eq:henon}. Then there exists $p\in \partial B_1(0)$ such that $u$ is foliated Schwarz symmetric with respect to $p$.
\item Assume that $\Omega\subset \R^N$, $N \geq 2$, is either a ball or an annulus centred at the origin. Then there exists $\delta_0>0$ such that, if $\alpha \in [0,\delta_0]$, then any least energy nodal solution $u$ of \eqref{eq:henon} is such that $u$ is non radially symmetric.
\end{enumerate}
\end{corollary}

Items i) and ii) are known cf. \cite{BartschWethWillem, CastroCossioNeuberger,PacellaWeth}. To our knowledge, Corollary \ref{th:eq.henon} iii) is new, though expected from a perturbation analysis for small $\alpha$. However, we stress that the approach of \cite{AftalionPacella} to symmetric breaking of any least energy nodal solution of the autonomous equation
\[
-\Delta u = f(u) \quad \mbox{ in } \Omega, \qquad u =0 \mbox{ on } \partial \Omega,
\]
where $\Omega \con \R^N$, $N \geq 2$, that stands either for a ball or an annulus centred at the origin, cannot be extended to the non autonomous case. In particular it cannot be extended to the H\'enon equation \eqref{eq:henon}. We expect however that least energy nodal solutions of \eqref{eq:henon} are non radial. We provide a proof for $\alpha$ small whereas this should follow from an asymptotic analysis as in \cite{SmetsSuWillem} for $\alpha$ large. The general case seems more delicate. 

The structure of the paper is the following. In Section \ref{sec:existence}, we introduce the variational setting corresponding to the dual method and prove Theorem \ref{thm:main1}, showing the existence of a least energy nodal solution, providing as well alternative characterizations of the level $c_{\rm nod}$. In Section \ref{sec:SchwarzFoliated} we prove the Schwarz foliated symmetry of these solutions when $\Omega$ is a radial bounded domain, namely Theorem \ref{thm:main2}. Finally in Section \ref{eq:Symmetry} we provide some examples of symmetry breaking, proving Theorems \ref{thm:symmetrybreak_close_diag} and \ref{lemma:u=v}.

\section{Existence of a least energy nodal level}\label{sec:existence}

Let us now introduce in a precise way the variational setting corresponding to the dual method. Given $r \geq 1$ and $\gamma\geq0$, we denote
\[
L^r(\Omega,|x|^{-\gamma}):=\{u:\Omega\to \R \text{ measurable}:\ \int_\Omega |u|^r|x|^{-\gamma}\, dx<\infty\},
\]
which is a Banach space equipped with the norm
\[
\|u\|_{r,\gamma}:= \left(\int_\Omega |u|^r|x|^{-\gamma}\, dx\right)^{1/r}.
\]
Observe that, since $\Omega$ is bounded and $\gamma\geq0$, we have the inclusions $L^r(\Omega,|x|^{-\gamma})\subset L^r(\Omega)$, where the last is the usual $L^r(\Omega)\,$- space. In fact, it is easy to check that there exists a constant $C(\Omega)$ such that
\begin{equation}\label{eq:estimate_between_L^r}
\|u\|_r\leq C(\Omega)^\frac{\gamma}{r}\|u\|_{r,\gamma} \qquad \forall u\in L^r(\Omega,|x|^{-\gamma}),\ r>1,\ \gamma\geq 0.
\end{equation}

In an informal basis, the method consists in taking the inverse of the Laplace operator, rewriting the system as
\[
(-\Delta)^{-1}(|x|^\beta |v|^{q-1}v)=u,\qquad (-\Delta)^{-1}(|x|^\alpha|u|^{p-1}u)=v.
\]
and defining $w_1=|x|^\alpha |u|^{p-1}u$, $w_2=|x|^\beta |v|^{q-1}v$, which leads to
\begin{equation}\label{eq:sistema_com_Delta-1}
(-\Delta)^{-1}w_2=|x|^{-\frac{\alpha}{p}}|w_1|^{\frac{1}{p}-1}w_1,\qquad (-\Delta)^{-1}w_1=|x|^{-\frac{\beta}{q}}|\omega_2|^{\frac{1}{q}-1}w_2.
\end{equation}

We will work in the product space
\[
X:=L^{\frac{p+1}{p}}(\Omega, |x|^{-\frac{\alpha}{p}})\times L^{\frac{q+1}{q}}(\Omega, |x|^{-\frac{\beta}{q}}),
\]
\[ 
\|(w_1,w_2)\|:=\|w_1\|_{\frac{p+1}{p},\frac{\alpha}{p}}+ \|w_2\|_{\frac{q+1}{q},\frac{\beta}{q}}\qquad \forall \, w=(w_1,w_2)\in X
  \]
and will use the map $T: X\to L^1(\Omega)$  given by
 \[
 Tw=w_1Kw_2+w_2Kw_1 \qquad w=(w_1,w_2)\in X
 \]
  where, with some abuse of notations,  $K$ denotes the inverse  of the minus Laplace operator with zero Dirichlet boundary condition. We observe that we use the same letter $K$ to denote both the operators $K_{(p+1)/p}: L^{\frac{p+1}{p}}(\Omega)\to  W^{2,\frac{p+1}{p}}(\Omega)\cap W_0^{1,\frac{p+1}{p}}(\Omega)$ and $K_{(q+1)/q}: L^{\frac{q+1}{q}}(\Omega)\to W^{2,\frac{q+1}{q}}(\Omega)\cap W_0^{1,\frac{q+1}{q}}(\Omega)$. Later on, we will use the fact that
 \begin{equation}\label{eq:K_tu=K_su}
  u\in L^{t}(\Omega)\cap L^{s}(\Omega)\Rightarrow K_t u = K_s u \qquad (t,s>1),
  \end{equation}
  which is a consequence of the uniqueness of the Dirichlet problem.
 Thanks to the subcriticality, namely the second inequality at \eqref{eq:subcritical}, we have compact embeddings
  \[
   W^{2,\frac{q+1}{q}}(\Omega)\subset L^{p+1}(\Omega), \qquad W^{2,\frac{p+1}{p}}(\Omega)\subset L^{q+1}(\Omega).
  \]
This, together with standard elliptic estimates, yields
\begin{align*}
\int_\Omega \left| w_1 Kw_2 \right|\, dx &\leq \|w_1\|_{\frac{p+1}{p}}\|Kw_2\|_{p+1}\leq C_1\|w_1\|_{\frac{p+1}{p}}\|K w_2\|_{W^{2,\frac{q+1}{q}}}\\
					   &\leq C_2 \|w_1\|_{\frac{p+1}{p}} \|w_2\|_{\frac{q+1}{q}}\leq C_3 \|w_1\|_{\frac{p+1}{p},\frac{\alpha}{p}} \|w_2\|_{\frac{q+1}{q},\frac{\beta}{q}},
\end{align*}
and an analogous estimate holds for $ w_2 K w_1$.
Thus
  \begin{equation}\label{eq:maj_for_Tw}
  \int_{\Omega} |Tw| \, dx \leq c \|w_1\|_{\frac{p+1}{p},\frac{\alpha}{p}} \|w_2\|_{\frac{q+1}{q},\frac{\beta}{q}}\qquad  \forall \, w=(w_1,w_2)\in X.
  \end{equation}
  Also, using integration by parts,
  \begin{equation}\label{eq:w_1Kw_2=Kw_1w_2}
  \int_{\Omega}w_1 Kw_2 \, dx= \int_{\Omega} w_2\, Kw_1\, dx   \qquad  \forall \, w=(w_1,w_2)\in X.
  \end{equation}
  
  Now, let $I: X\to \IR$ be the functional defined by
  \begin{multline}\label{energy I}
I(w_1,w_2)=\frac{p}{p+1}\int_{\Omega} |w_1|^{\frac{p+1}{p}}|x|^{-\frac{\alpha}{p}}\, dx+\frac{q}{q+1}\int_{\Omega} |w_2|^{\frac{q+1}{q}}|x|^{-\frac{\beta}{q}}\, dx\\-\frac{1}{2}\int_{\Omega}Tw\,dx.
\end{multline}
It is easy to see that $I$ is a $C^1$ functional and, thanks to \eqref{eq:w_1Kw_2=Kw_1w_2}, that its derivative is given by
\begin{multline*}
I'(w)(\varphi,\psi)= \int_{\Omega} |w_1|^{\frac{1}{p}-1} w_1\varphi |x|^{-\frac{\alpha}{p}}\, dx+\int_{\Omega} |w_2|^{\frac{1}{q}-1} w_2\psi |x|^{-\frac{\beta}{q}}\, dx\\-\int_{\Omega}(\varphi Kw_2+\psi Kw_1)\, dx ,
\end{multline*}
for every $(w_1,w_2), (\phi,\varphi)\in X$. 
In particular, $(w_1,w_2)$ is a critical point of $I$ if and only if \eqref{eq:sistema_com_Delta-1} holds, and so $(u, v):=(|x|^{-\frac{\alpha}{p}}|w_1|^{\frac{1}{p}-1}w_1, |x|^{-\frac{\beta}{q}}|w_2|^{\frac{1}{q}-1}w_2)$ is a strong solution of the original system \eqref{eq:mainsystem}. In this case we have that
\begin{equation}\label{eq:comparision_between_I_E}
I(w_1,w_2)=E(u_1,u_2)=\frac{pq-1}{(p+1)(q+1)}\int_\Omega |u|^{p+1}|x|^\alpha\, dx>0.
\end{equation}

Working in this framework, we can rewrite the least energy nodal level as
\[
c_{\rm nod}=\inf\{I(w_1,w_2):\ w_1^\pm,w_2^\pm\not \equiv 0,\ I'(w_1,w_2)=0\}.
\]
In the following, we will adapt some of the ideas of \cite{AlvesSoares, ClementvanderVorst} into our situation. A first novelty in our arguments consists in introducing the following constants $\lambda$ and $\mu$ (in view of evening the different powers of $w_1$ and $w_2$ in the functional $I$), as well as the introduction of the map $\theta$, in view of obtaining Proposition \ref{prop:unique_maximum} ahead. Let
 \[
 \lambda:=\frac{2p(q+1)}{p+q+2pq}, \qquad \mu=\frac{2q(p+1)}{p+q+2pq},
 \]
 so that
 \[
 \gamma:=\lambda \, \frac{p+1}{p}=\mu\, \frac{q+1}{q} = \frac{2(p+1)(q+1)}{p+ q + 2pq}\in (1,2) \quad \mbox{ and } \qquad \lambda+\mu=2.
 \]
We start by introducing the Nehari type set:
\begin{multline*}
\Ncal_{\rm nod}=\{(w_1, w_2)\in X: \ w_1^{\pm}\not \equiv 0, \ w_2^{\pm}\not \equiv 0 \ \text{ and } \\ I'(w)(\lambda w_1^+,\mu w_2^+)=I'(w)(\lambda w_1^-,\mu w_2^-)=0\}
\end{multline*}
and the level
\[
\tilde c_{\rm nod}=\inf_{w\in \Ncal_{\rm nod}}I(w),
\]
which, we will check later on that, coincides with $c_{\rm nod}$ and that $I$ is positive on $\Ncal_{\rm nod}$; cf. Theorem \ref{thm:levelsequiv} and \eqref{functionalpositivenehari} respectively. The study of this problem will be done by means of a fiber-type map: given $(w_1, w_2) \in X$ such that $w_1^{\pm}\not \equiv 0$ and $w_2^{\pm} \not \equiv 0$, define $\theta=\theta_w:[0, \infty)\times [0, \infty)\to \R$ by
\begin{equation}\label{eq:definition_of_theta}
\theta(t,s)=I(t^\lambda w_1^+-s^\lambda w_1^-,t^\mu w_2^+-s^\mu w_2^-).
\end{equation}
Observe that if $t,s > 0$, then 
\begin{equation}\label{eq:criticalpoint_theta}
\nabla \theta(t,s)=(0,0)\quad \text{ if and only if } \quad (t^\lambda w_1^+-s^\lambda w_1^-,t^\mu w_2^+-s^\mu w_2^-)\in \Ncal_{\rm nod}.
\end{equation}
In what follows it will be important to prove that $\theta_w$ admits a critical point. However, this seems not to hold for every $w\in X$. Indeed, for some $w$ it turns out that $\theta_w$ has no critical points and its supremum is plus infinity. For that reason, we introduce the following auxiliary set
\begin{align*}
\Ncal_0&=\left\{w\in X:\ \begin{array}{c} \lambda \int_\Omega w_1^+ Kw_2\, dx+\mu \int_\Omega w_1 Kw_2^+\, dx >0\\[0.1cm] 
			\lambda \int_\Omega w_1^- Kw_2\, dx+\mu \int_\Omega w_1 Kw_2^-\, dx <0  \end{array}\right\}\\
	   &=\left\{w\in X:\ \lambda C_1+\mu C_2<2B^+,\mu C_1+\lambda C_2<2B^-\right\},
\end{align*}
where we have used the fact that $\lambda+\mu=2$, and the notations
\begin{equation}\label{eq:Bmaismenos}
B^+=\int_\Omega w_1^+ Kw_2^+\, dx>0,\qquad B^-=\int_\Omega w_1^-Kw_2^-\, dx>0, 
\end{equation}
and
\[
C_1=\int_\Omega w_1^+ K w_2^-\, dx>0,\qquad C_2=\int_\Omega w_1^- K w_2^+\, dx>0.
\]
We point out that, by the strong maximum principle, if $w\in {\Ncal}_0$ then $w_1^\pm\not \equiv 0$, $w_2^\pm\not \equiv 0$.
\begin{lemma}\label{lemma:N_nonempty} 
The set $\Ncal_0$ is nonempty and $\Ncal_{\rm nod}\subseteq \Ncal_0$.
\end{lemma}
\begin{proof}
{\it Claim 1} -- $\Ncal_0\neq \emptyset$. Let $\varphi_2$ be the second eigenfunction of $-\Delta$ with zero Dirichlet boundary condition, and denote its eigenvalue by $\lambda_2$. Then $(\varphi_2,\varphi_2)\in \Ncal_0$, as 
\[
\lambda \int_\Omega \varphi_2^+ K\varphi_2\, dx+\mu \int_\Omega \varphi_2 K\varphi_2^+\, dx=\frac{2}{\lambda_2}\int_\Omega (\varphi_2^+)^2>0
\]
and
\[
\lambda \int_\Omega \varphi_2^- K\varphi_2\, dx+\mu \int_\Omega \varphi_2 K\varphi_2^-\, dx=-\frac{2}{\lambda_2}\int_\Omega (\varphi_2^-)^2<0.
\]
{\it Claim 2} -- $\Ncal_{\rm nod}\subseteq \Ncal_0$. This is an immediate consequence of the equalities that define $\Ncal_{\rm nod}$. Indeed, since $w\neq 0$,
\begin{multline*}
\lambda \int_\Omega w_1^+ Kw_2\, dx+\mu \int_\Omega w_1 Kw_2^+\, dx\\
=\lambda \int_\Omega |w_1^+|^{\frac{p+1}{p}}|x|^{-\frac{\alpha}{p}}\, dx+\mu \int_\Omega |w_2^+|^{\frac{q+1}{q}}|x|^{-\frac{\beta}{q}}\, dx>0
\end{multline*}
and
\begin{multline*}
\lambda \int_\Omega w_1^- Kw_2\, dx+\mu \int_\Omega w_1 Kw_2^-\, dx\\
=-\lambda \int_\Omega |w_1^-|^{\frac{p+1}{p}}|x|^{-\frac{\alpha}{p}}\, dx+\mu \int_\Omega |w_2^-|^{\frac{q+1}{q}}|x|^{-\frac{\beta}{q}}\, dx<0. \qedhere
\end{multline*}
\end{proof}

Let us now study in detail the map $\theta$, following a standard procedure. We recall that $\theta = \theta_{w}$ for $w = (w_1, w_2) \in X$ such that $w_1^{\pm}\neq 0$ and $w_2^{\pm} \neq 0$. After a few computations, we can rewrite it as
\[
\theta(t,s)=A^+t^\gamma  + A^-s^\gamma  - B^+ t^2-B^-s^2+C_1 t^\lambda s^\mu + C_2t^\mu s^\lambda,
\]
with
\begin{equation}\label{eq:Amaismenos}
A^\pm= \frac{p}{p+1}\int_\Omega |w_1^\pm|^{\frac{p+1}{p}}|x|^{-\frac{\alpha}{p}}\, dx+\frac{q}{q+1}\int_\Omega |w_2^\pm|^{\frac{q+1}{q}}|x|^{-\frac{\beta}{q}}\, dx>0.
\end{equation}

Our goal is to prove that $\theta_w$, for $w\in \Ncal_0$, admits a unique critical point, which is a global maximum, attained at a pair with positive components. We divide the proof of this fact in several lemmas.

\begin{lemma}\label{lemma:global_maximum}
Let $w\in \Ncal_0$ and take $\theta=\theta_w$. Then $\theta$ has a global maximum at some $(t_0,s_0)$ with $t_0,s_0>0$. Moreover, every local maximum must have positive components.
\end{lemma}
\begin{proof}
Young's inequality yields
\[
t^\lambda s^\mu\leq \frac{\lambda t^2}{2}+\frac{\mu s^2}{2} \quad \text{ and } \quad t^\mu s^\lambda \leq \frac{\mu t^2}{2}+\frac{\lambda s^2}{2} \qquad \forall \, t,s\geq 0,
\]
and thus
\[
\theta(t,s)\leq A^+ t^\gamma + A^-s^\gamma  + \left(\frac{1}{2}(\lambda C_1+\mu C_2)-B^+\right) t^2+\left(\frac{1}{2}(\mu C_1+\lambda C_2)-B^-\right) s^2.
\]
As $w\in \Ncal_0$, the coefficients of the quadratic terms are negative, hence $\theta(t,s)\to -\infty$ as $|s|+|t|\to +\infty$, and $\theta$ admits a global maximum $(t_0,s_0)$ with \emph{nonnegative} components.

To conclude, let us prove that it cannot happen that either $t_0=0$ or $s_0=0$. This is an immediate consequence of the fact that
\begin{align}
\theta(t,s) &= A^+ t^\gamma -B^+ t^2 + s^\gamma (A^- - B^-s^{2-\gamma})+C_1 t^\lambda s^\mu + C_2 t^\mu s^\lambda >\theta(t,0)
\end{align}
for $s>0$ sufficiently small. Analogously, $\theta(t,s)>\theta(0,s)$ for $t>0$ sufficiently small.
\end{proof}

\begin{lemma}\label{lemma:critical->maximum}
If $(t,s)$ is a critical point of $\theta$ with $t,s>0$, then $(t,s)$ is a non degenerate local maximum. 
\end{lemma}
\begin{proof}
If $(t,s)$ is a critical point of $\theta$, then
\[
2B^+=\gamma A^+ t^{\gamma-2}+\lambda C_1 t^{\lambda-2}s^\mu + \mu C_2 t^{\mu-2}s^\lambda,
\]
and
\[
2B^-=\gamma A^- s^{\gamma-2}+\mu C_1 t^{\lambda}s^{\mu-2} + \lambda C_2 t^{\mu}s^{\lambda-2}.
\]
Thus
\begin{align*}
\theta_{tt}(t,s) &=\gamma(\gamma-1)A^+ t^{\gamma-2}-2B^+ + \lambda (\lambda-1)C_1 t^{\lambda-2} s^{\mu}+\mu (\mu-1)C_2 t^{\mu-2}s^\lambda\\
	      &=\gamma(\gamma-2) A^+ t^{\gamma-2}+\lambda(\lambda-2)C_1 t^{\lambda-2} s^\mu+\mu(\mu-2)C_2 t^{\mu-2}s^\lambda<0,
\end{align*}
\begin{align*}
\theta_{ss}(t,s) &=\gamma(\gamma-1)A^- s^{\gamma-2}-2B^- + \mu (\mu-1)C_1 t^{\lambda} s^{\mu-2}+\lambda (\lambda-1)C_2 t^{\mu}s^{\lambda-2}\\
	      &=\gamma(\gamma-2) A^- s^{\gamma-2}+\mu(\mu-2)C_1 t^\lambda s^{\mu-2}+\lambda(\lambda-2)C_2 t^{\mu}s^{\lambda-2}<0,
\end{align*}
and 
\begin{align*}
\theta_{ts}(t,s)=\lambda \mu C_1 t^{\lambda-1}s^{\mu-1}+\lambda \mu C_2 t^{\mu-1}s^{\lambda-1}.
\end{align*}
The proof is complete as soon as we prove that $\theta_{ts}^2<\theta_{tt}\theta_{ss}$, which is equivalent to
\begin{align}\label{eq:huge_inequality}
\nonumber \lambda^2 \mu^2 &C_1^2 t^{2(\lambda-1)}s^{2(\mu-1)}+\lambda^2\mu^2 C_2^2 t^{2(\mu-1)}s^{2(\lambda-1)}+2\lambda^2\mu^2 C_1 C_2\\
\nonumber     < &\lambda(\lambda-2)\mu(\mu-2) C_1^2 t^{2(\lambda-1)}s^{2(\mu-1)}+\lambda(\lambda-2)\mu(\mu-2)C_2^2 t^{2(\mu-1)}s^{2(\lambda-1)}\\
\nonumber	&+[\lambda^2(\lambda-2)^2+\mu^2(\mu-2)^2]C_1C_2+\gamma^2 (\gamma-2)^2 A^+A^- t^{\gamma-2}s^{\gamma-2}\\
\nonumber	&+\gamma(\gamma-2)\mu(\mu-2)A^+ C_1 t^{\lambda+\gamma-2}s^{\mu-2}+\gamma(\gamma-2)\lambda(\lambda-2) A^- C_1 t^{\lambda-2}s^{\mu+\gamma-2}\\
	&+\gamma(\gamma-2)\lambda(\lambda-2)A^+C_2 t^{\mu+\gamma-2}s^{\lambda-2}+\mu(\mu-2)\gamma(\gamma-2)C_2A^- s^{\lambda+\gamma-2}t^{\mu-2}.
\end{align}
Now $\lambda+\mu=2$ is equivalent to $\lambda^2\mu^2=\lambda(\lambda-2)\mu(\mu-2)$, which in turn implies that
\[
2\lambda^2\mu^2\leq \lambda^2(\lambda-2)^2+\mu^2(\mu-2)^2. 
\] 
Since the last five terms in \eqref{eq:huge_inequality} are positive, combining all of these we prove that the desired inequality holds true.
\end{proof}

\begin{proposition}\label{prop:unique_maximum}
Let $w\in \Ncal_0$. The map $\theta_w$ admits a unique critical point $(t_0,s_0)$ with $t_0,s_0>0$, which corresponds to its unique global maximum. Moreover, the pair $(t_0,s_0)$ can be characterized as being the (unique) solution to the system:
\begin{equation}\label{eq:system_t_s}
\begin{cases}
2B^+=\gamma A^+ t^{\gamma-2}+\lambda C_1 (s/t)^\mu + \mu C_2 (s/t)^\lambda, \\
2B^-=\gamma A^- s^{\gamma-2}+\mu C_1 (t/s)^{\lambda} + \lambda C_2 (t/s)^{\mu}.
\end{cases}
\end{equation}

\end{proposition}

\begin{proof}
The only think left to prove is the uniqueness statement. One could argue as in \cite[Proposition 3.2]{TavaresWeth}, but here instead we present a shorter argument, which makes use of the Poincar\'e-Hopf Theorem \cite{Milnor}.
Recall that this result states that given $M$ a smooth manifold with boundary, and $X$ a vector field having only isolated zeros $x_i$ ($i\in I$) and such that it points outward on $\partial M$, then
\[
\chi(M)=\sum_{i_I} \text{index}(X,x_i),
\]
where $\chi(\cdot)$ is the Euler characteristic. Let $(t_0,s_0)$ be a global maximum as in Lemma \ref{lemma:global_maximum}. Take $M$ to be a bounded regular set containg $(t_0,s_0)$, which coincides with the square $[\eps,L]\times[\eps,L]$ expect at the corners, where it is smooth. Then $\chi(M)=1$, and $X=-\nabla \theta$ points outward on $\partial M$ for sufficiently small $\eps$ and $L$ large enough. Lemma \ref{lemma:critical->maximum} on the other hand implies that index$(-\nabla \theta, (s,t))=1$ at each critical point $(s,t)$, and thus we prove that $(t_0,s_0)$ is indeed the unique critical point of $\theta_w$.
\end{proof}

\begin{lemma}\label{lemma:minimum_of_I_is_critical}
Let $w\in \Ncal_{\rm nod}$ be such that $I(w)=\tilde c_{\rm nod}$. Then $I'(w)=0$.
\end{lemma}
\begin{proof}
We will argue as in \cite{deValeriolaWillem}; see also \cite{BartschWethWillem, ContiTerraciniVerzini}. 

{\it Step 1} -- Assume by contradiction that $I'(w)\neq 0$. Then there exists $v\in X$ such that $I'(w)v=-2$. By continuity, there exists a small $\eps>0$ such that
\[
I'( (t^{\lambda}w_1^+-s^{\lambda}w_1^-, t^\mu w_2^+-s^\mu w_2^-)+rv)v<-1\qquad \forall \, 0\leq r\leq \eps, \ |t-1|,|s-1|\leq \eps.
\]
Set $D:=[1-\eps,1+\eps]\times [1-\eps,1+\eps]$. We fix a smooth function $\eta: D \to [0,\eps] $ such that $\eta(1,1)=\eps$ and $\eta=0$ on $\partial D$, and denote
\[
h(t,s)= (h_1(t,s), h_2(t_s)): =(t^{\lambda}w_1^+-s^{\lambda}w_1^-, t^\mu w_2^+-s^\mu w_2^-)+\eta(t,s)v,
\]
$$
H(t,s)=\Big( I'(h(t,s))(\lambda h_1^+(t,s),\mu h_2^+(t,s)), I'(h(t,s))(\lambda h_1^-(t,s),\mu h_2^-(t,s))\Big).$$
By possibly taking a smaller $\eps$, we can insure by continuity that
$$
h(t,s)\in \Ncal_0 \qquad \forall \, (t,s)\in D.$$

{\it Step 2} -- We claim that there exists  $(t_0,s_0)\in D$ such that $H(t_0,s_0)=(0,0)$. To prove this, we use the classical Miranda's Theorem \cite{Miranda}. We will need to compute $H$ on $\partial D$, where as $\eta=0$,
\[
H(t,s)\!=\!\left(\begin{array}{c} I'(t^\lambda w_1^+ - s^\lambda w_1^- , t^\mu w_2^+ - s^\mu w_2^-)(\lambda t^\lambda w_1^+,\mu t^\mu w_2^+)\\
				I'(t^\lambda w_1^+ - s^\lambda w_1^- , t^\mu w_2^+ - s^\mu w_2^-)(\lambda s^\lambda w_1^-,\mu s^\mu w_2^-)
				\end{array}\right)\!=\!\left( \begin{array}{c} t\, \theta_t(t,s) \\ -s\, \theta_s(t,s) \end{array}\right)\!.
\]
We have $\nabla \theta(1,1)=(0,0)$, which tells us that
\[
\gamma A^+-2 B^+ + \lambda C_1+\mu C_2=0,\qquad \gamma A^- - 2B^- +\mu C_1+\lambda C_2=0.
\]
For $s\in [1-\eps,1+\eps]$ we have that, if $t=1+\eps$, then
\begin{align*}
\theta_t &=\gamma A^+ (1+\eps)^{\gamma-1}-2B^+(1+\eps)+\lambda C_1(1+\eps)^{\lambda-1}s^\mu + \mu C_2 (1+\eps)^{\mu-1}s^{\lambda}\\
		&\leq \gamma A^+ (1+\eps)[(1+\eps)^{\gamma-2}-1]<0;
\end{align*}
while if $t=1-\eps$
\[
\theta_t \geq \gamma A^+ (1-\eps)((1-\eps)^{\gamma-2}-1)>0.
\]
Analogously, for $t\in [1-\eps,1+\eps]$, we have
\[
\theta_s<0 \quad \text{for } s=1+\eps,\qquad \theta_s>0 \quad \text{for } s=1-\eps
\]
and the claim follows.

{\it Conclusion} -- By the previous point, which shows that  $h(t_0,s_0)\in {\Ncal}_{{\rm nod}} $, it follows that
\begin{align*}
\tilde c_{{\rm nod}}\leq& I(h(t_0,s_0))\\
=& I(t_0^{\lambda}w_1^+-s_0^{\lambda}w_1^-, t_0^\mu w_2^+-s_0^\mu w_2^-) \\
	&+\int_0^{\eta(t_0,s_0)}I'( (t_0^{\lambda}w_1^+-s_0^{\lambda}w_1^-, t_0^\mu w_2^+-s_0^\mu w_2^-)+rv)v\, dr\\
\leq & \theta_w(t_0,s_0) -\eta(t_0,s_0)\leq  \theta_w(1,1) -\eta(t_0,s_0) = \tilde c_{{\rm nod}}-\eta(t_0,s_0),
\end{align*}
and so $\eta(t_0,s_0)=0$ and, in particular, $\theta_w(t_0,s_0)=\theta_w(1,1)$. By the uniqueness of maximum provided by Proposition \ref{prop:unique_maximum} we must have $(t_0,s_0)=(1,1)$ while, by construction, $\eta(1,1)=\eps>0$, a contradiction.
\end{proof}

Our purpose in the remainder of this section is to prove the following result.

\begin{theorem}\label{thm:levelsequiv}
The number $\tilde c_{\rm nod}$ is attained by a function $w\in \Ncal_{\rm nod}$. Moreover,
\begin{align}\label{align:happy_end}
c_{\rm nod} &= \tilde c_{\rm nod} =\inf_{w\in \Ncal_0} \sup_{t,s> 0} I(t^\lambda w_1^+-s^\lambda w_1^-,t^\mu w_2^+-s^\mu w_2^-)>0.
\end{align}
\end{theorem}

Observe that this implies our first main result, namely Theorem \ref{thm:main1}. From \eqref{eq:criticalpoint_theta}, Lemma \ref{lemma:N_nonempty} and Proposition \ref{prop:unique_maximum}, we have
\begin{equation}\label{eq:goal_easyequality}
\tilde c_{\rm nod} =\inf_{w\in \Ncal_0} \sup_{t,s> 0} I(t^\lambda w_1^+-s^\lambda w_1^-,t^\mu w_2^+-s^\mu w_2^-).
\end{equation}
So the only thing left to prove is that $\tilde c_{\rm nod}$ is achieved. To this aim, we rely on an indirect argument. The difficulty of a direct approach is that the weak limit of a minimizing sequence of $\tilde c_{\rm nod}$ does not belong necessarily to ${\Ncal_{0}}$, and we cannot project it in $\Ncal_{\rm nod}$. Indeed, even if we can bound from below the norms of the positive and negative parts of a minimizing sequence, the weak convergence in $X$ does not imply the weak convergence of the positive and negative part of the sequence. Observe also that the lack of regularity of the nodal Nehari set makes rather tricky the use of Ekeland's principle to build a Palais-Smale sequence from a minimizing sequence. 

\medbreak

To overpass this difficulty, we regularize the problem by introducing the auxiliary functional 
$I_{\varepsilon}: \widetilde X\to \IR$ defined by
  \begin{align}\label{energy Ieps}
I_{\varepsilon}(w_1,w_2)= & 
\frac{p}{p+1}\left(\varepsilon \int_{\Omega} |\nabla w_1|^{\frac{p+1}{p}}\, dx + \int_{\Omega} |w_1|^{\frac{p+1}{p}}|x|^{-\frac{\alpha}{p}}\, dx\right) \\  & +\frac{q}{q+1}\left(\varepsilon \int_{\Omega} |\nabla w_2|^{\frac{q+1}{q}}\, dx +\int_{\Omega} |w_2|^{\frac{q+1}{q}}|x|^{-\frac{\beta}{q}}\, dx\right) -\frac{1}{2}\int_{\Omega}Tw\,dx,
\end{align}
where $\varepsilon>0$ and $\widetilde X:= (W^{1,\frac{p+1}{p}}(\Omega)\times W^{1,\frac{q+1}{q}}(\Omega))\cap X$. 

Such a regularization is a standard approach to regularize non uniformly elliptic operators, such as the curvature operator, see for instance \cite{Temam,LadyUral,BonHabObOm}. It is surprisingly useful in our context to regularize a zero order term. Actually, the main utility of this approach is that it somehow provides a regularized minimizing sequence which solves an approximating system. The key point is then that the regularization does not affect the geometry of the original functional, while the presence of the gradient terms give rise to Euler-Lagrange equations in which we can pass to the limit. For that, we will exploit the fact that $\widetilde X$ is dense in $X$.

\begin{lemma}
Let $\widetilde X= (W^{1,\frac{p+1}{p}}(\Omega)\times W^{1,\frac{q+1}{q}}(\Omega))\cap X$ be endowed with the norm
\[
\| (w_1, w_2) \|  = \| \nabla w_1\|_{\frac{p+1}{p}} + \| \nabla w_2 \|_{\frac{q+1}{q}} + \| w_1\|_{\frac{p+1}{p}, \frac{\alpha}{p}} + \| w_2\|_{\frac{q+1}{q}, \frac{\beta}{p}}.
\]
Then $\widetilde X$ is a reflexive Banach space which is continuously embedded in $W^{1,\frac{p+1}{p}}(\Omega)\times W^{1,\frac{q+1}{q}}(\Omega)$. Moreover $\widetilde X$ is a dense subspace of $X$. 
\end{lemma}
\begin{proof}
The first two statements are obvious. To prove the density of $\widetilde X$ in $X$,  first observe that
\[
T : L^{\frac{p+1}{p}}(\Omega) \times L^{\frac{q+1}{q}}(\Omega) \longrightarrow L^{\frac{p+1}{p}}(\Omega, |x|^{- \frac{\alpha}{p}}) \times L^{\frac{q+1}{q}}(\Omega, |x|^{- \frac{\beta}{q}}),
\]
defined by $T(f,g) = (f|x|^{\frac{\alpha}{p+1}},g|x|^{\frac{\beta}{q+1}} )$, is an isometric isomorphism. For each $\delta> 0$, fix $\varphi_{\delta} \in C_c^{\infty}(\R^N)$ such that
\[
0 \leq \varphi_{\delta} \leq 1, \quad \varphi_{\delta}(x) = 1 \ \ \text{if} \ \ |x|\geq 2 \delta \ \ \text{and} \ \ \varphi_{\delta}(x) = 0 \ \ \text{if} \ \ |x| \leq \delta.
\]
Then observe that $A = \{ (f \varphi_{\delta},g \varphi_{\delta}); f,g \in C_c^{\infty}(\Omega), \ \ \delta>0  \}$ is dense in $L^{\frac{p+1}{p}}(\Omega) \times L^{\frac{q+1}{q}}(\Omega)$ and that $T(A) \subset \widetilde X \subset X$.
\end{proof}

One can check easily that the functional $I_\eps$ belongs to $C^1(\widetilde X)$, and
\begin{multline*}
I_{\eps}'(w)(\varphi,\psi)= \eps \int_\Omega |\nabla w_1|^{\frac{1}{p}-1}\nabla w_1\cdot \nabla \varphi\, dx+   \int_{\Omega} |w_1|^{\frac{1}{p}-1} w_1\varphi |x|^{-\frac{\alpha}{p}}\, dx   \\
+ \eps\int_\Omega   |\nabla w_2|^{\frac{1}{q}-1}\nabla w_2\cdot \nabla \psi \, dx + \int_{\Omega} |w_2|^{\frac{1}{q}-1} w_2\psi |x|^{-\frac{\beta}{q}}\, dx-\int_{\Omega}(\varphi Kw_2+\psi Kw_1)\, dx ,
\end{multline*}
for every $(w_1,w_2), (\phi,\varphi)\in \widetilde X$. 
In particular, $(w_1,w_2)$ is a critical point of $I_\eps$ if and only if
\begin{equation}\label{eq:perturbed_system_eps}
\begin{cases}
-\eps \text{div}(|\nabla w_1|^{\frac{1}{p}-1}\nabla w_1)+ |w_1|^{\frac{1}{p}-1} w_1 |x|^{-\frac{\alpha}{p}}= Kw_2\\
-\eps \text{div}(|\nabla w_2|^{\frac{1}{q}-1}\nabla w_2)+ |w_2|^{\frac{1}{q}-1} w_2 |x|^{-\frac{\alpha}{p}}= Kw_1
\end{cases}
\end{equation}
in $\widetilde X^\ast$. Define now
\begin{align*}
\widetilde \Ncal_0&=\left\{w\in \widetilde X:\ \begin{array}{c} \lambda \int_\Omega w_1^+ Kw_2\, dx+\mu \int_\Omega w_1 Kw_2^+\, dx >0\\[0.1cm] 
			\lambda \int_\Omega w_1^- Kw_2\, dx+\mu \int_\Omega w_1 Kw_2^-\, dx <0  \end{array}\right\}
\end{align*}
and, for each $\varepsilon>0$,
\begin{multline*}
\Ncal^\eps_{\rm nod}=\{(w_1, w_2)\in \widetilde X: \ w_1^{\pm}\not \equiv 0, \ w_2^{\pm}\not \equiv 0 \ \text{ and } \\ I_\eps'(w)(\lambda w_1^+,\mu w_2^+)=I_\eps'(w)(\lambda w_1^-,\mu w_2^-)=0\}.
\end{multline*}
Observe that
\[
\Ncal_{\rm nod}^\eps\subset \widetilde \Ncal_0 \subset \Ncal_0.
\]

In the statement of the next lemma and its proof, we keep the definitions of $A^{\pm}$ and $B^{\pm}$ as given in \eqref{eq:Amaismenos} and \eqref{eq:Bmaismenos}.

\begin{lemma}\label{lemma:uniquemax_eps}
Let $w\in \widetilde \Ncal_0$. Then the map $\R^+\times \R^+\to \R$ defined by
\[
(t,s)\mapsto I_\eps(t^\lambda w_1^+-s^\lambda w_1^-,t^\mu w_2^+-s^\mu w_2^-)
\]
admits a unique critical point $(t_0,s_0)$, which is a global maximum. Moreover, the pair $(t_0,s_0)$ can be characterized univocally by the condition
\[
(t^\lambda w_1^+-s^\lambda w_1^-,t^\mu w_2^+-s^\mu w_2^-)\in \Ncal^\eps_{\rm nod}
\]
or equivalently through the system
\begin{equation}\label{eq:system_t_s_eps}
\begin{cases}
2B^+=&\displaystyle \gamma \left(A^+ +\eps\frac{p}{p+1}\int_\Omega |\nabla w_1^+|^\frac{p+1}{p}+\eps \frac{q}{q+1}\int_\Omega |\nabla w_2^+|^\frac{q+1}{q} \right) t^{\gamma-2}\\
        &+\lambda C_1 (s/t)^\mu + \mu C_2 (s/t)^\lambda, \\
2B^-=&\displaystyle \gamma \left(A^- +\eps \frac{p}{p+1}\int_\Omega |\nabla w_1^-|^\frac{p+1}{p}+\eps \frac{q}{q+1}\int_\Omega |\nabla w_2^-|^\frac{q+1}{q}\right) s^{\gamma-2}\\
&+\mu C_1 (t/s)^{\lambda} + \lambda C_2 (t/s)^{\mu}.
\end{cases}
\end{equation}
\end{lemma}
\begin{proof} Since the functional $I_\eps$ has exactly the same shape and geometry of $I$, it is enough to repeat the proofs of Lemmas \ref{lemma:N_nonempty}-\ref{lemma:critical->maximum} and of Proposition \ref{prop:unique_maximum}, replacing only $\Ncal_0$ and $\Ncal_{\rm nod}$ by $\widetilde \Ncal_0$ and $\Ncal^\eps_{\rm nod}$ respectively, and $A^\pm$ by
\[
A^\pm +\eps\frac{p}{p+1}\int_\Omega |\nabla w_1^\pm|^\frac{p+1}{p}+\eps \frac{q}{q+1}\int_\Omega |\nabla w_2^\pm|^\frac{q+1}{q}. \qedhere
\]
\end{proof}

Define the levels
\[
c^\eps_{\rm nod}=\inf\{I_\eps(w):\ w\in \widetilde X,\ w_1^\pm,w_2^\pm \not\equiv 0,\ I_\eps'(w)=0\}
\]
and
\[
\tilde c^\eps_{\rm nod}=\inf_{\Ncal_{\rm nod}^\eps} I_\eps.
\]
\begin{lemma}\label{lemma:minimum_of_Ieps_is_critical}
Given $\eps>0$, let $w\in \Ncal^\eps_{\rm nod}$ be such that $I_\eps(w)=\tilde c^\eps_{\rm nod}$. Then $I_\eps'(w)=0$.
\end{lemma}
\begin{proof}
The proof follows the lines of that of Lemma \ref{lemma:minimum_of_I_is_critical} with obvious changes as in the proof of Lemma \ref{lemma:uniquemax_eps}.
\end{proof}

\begin{proposition}\label{prop:levelsequiv_eps}
Given $\eps>0$, the number $c^\eps_{\rm nod}$ is attained by a function $w_\eps \in \Ncal^\eps_{\rm nod}$. Moreover, we have
\begin{equation}\label{eq:levelseps}
c^\eps_{\rm nod} = \tilde c^\eps_{\rm nod} =\inf_{w\in \widetilde \Ncal_0} \sup_{t,s> 0} I(t^\lambda w_1^+-s^\lambda w_1^-,t^\mu w_2^+-s^\mu w_2^-)>0.
\end{equation}
\end{proposition}
\begin{proof}
{\it Step 1 }-- $\Ncal^\eps_{\rm nod}$ is not empty for every $\eps>0$. This clearly follows from Lemma \ref{lemma:uniquemax_eps}. 

\medbreak

{\it Step 2 }-- boundedness and convergence of minimizing sequences. Let $(w_n)_{n}\subset \Ncal^\eps_{\rm nod}$ be a minimizing sequence for $\tilde c^\eps_\text{nod}$. Denote, by simplicity, 
\[
a_n=\int_\Omega | w_{1,n}|^{\frac{p+1}{p}}|x|^{-\frac{\alpha}{p}}\, dx, \qquad a_n^\pm =\int_\Omega |w_{1,n}^\pm|^{\frac{p+1}{p}}|x|^{-\frac{\alpha}{p}}\, dx
\]
and
\[
b_n=\int_\Omega | w_{2,n}|^{\frac{q+1}{q}}|x|^{-\frac{\beta}{q}}\, dx,\qquad   b_n^\pm =\int_\Omega |w_{2,n}^\pm|^{\frac{q+1}{q}}|x|^{-\frac{\beta}{q}}\, dx.
\]
One has
\begin{multline}\label{PS2}
\lambda \left(a_n^+ + \eps \int_\Omega |\nabla w_{1,n}^+|^\frac{p+1}{p} \, dx \right) +\mu \left(b_n^+ + \eps \int_\Omega |\nabla w_{2,n}^+|^\frac{q+1}{q} \, dx\right)\\
			= \lambda \int_{\Omega}w_{1,n}^{+}Kw_{2,n}\, dx+\mu \int_{\Omega}w_{2,n}^{+}Kw_{1,n}\, dx,
\end{multline}
\begin{multline}\label{PS3}
\lambda \left( a_n^-  + \eps \int_\Omega |\nabla w_{1,n}^-|^\frac{p+1}{p}\, dx \right)+\mu \left(b_n^- + \eps \int_\Omega |\nabla w_{2,n}^-|^\frac{q+1}{q} \, dx\right) \\
			= -\lambda \int_{\Omega}w_{1,n}^{-}Kw_{2,n}\, dx-\mu \int_{\Omega}w_{2,n}^{-}Kw_{1,n}\, dx.
\end{multline}
By adding \eqref{PS2} and \eqref{PS3}, we obtain
\begin{equation*}
\lambda \left(a_n + \eps\int_\Omega |\nabla w_{1,n}|^\frac{p+1}{p}\, dx\right)+\mu \left(b_n + \eps\int_\Omega |\nabla w_{2,n}|^\frac{q+1}{q} \, dx\right)= 2\int_{\Omega}w_{1,n}Kw_{2,n}\, dx
\end{equation*}
and we deduce that
\begin{multline}\label{functionalpositivenehari}
I_\eps(w_n)=\frac{(pq-1)p}{(p+1)(2pq+p+q)} \left(a_n + \eps\int_\Omega |\nabla w_{1,n}|^\frac{p+1}{p}\, dx\right)\\
				 +\frac{(pq-1)q}{(q+1)(2pq+p+q)}  \left(b_n + \eps\int_\Omega |\nabla w_{2,n}|^\frac{q+1}{q} \, dx \right)>0.
\end{multline}
Observe that this shows that $I_\eps$ is positive on $\Ncal^\eps_{\rm nod}$. Therefore $(w_n)_{n}$ is bounded in $\widetilde X$, and up to a subsequence, we have that $w_n\rightharpoonup w$ weakly in $\widetilde X$, strongly in $L^\frac{p+1}{p}(\Omega)\times L^\frac{q+1}{q}(\Omega)$. In particular, $(w_{1,n}^\pm,w_{2,n}^\pm)\to (w_1^\pm,w_2^\pm)$ in $L^\frac{p+1}{p}(\Omega)\times L^\frac{q+1}{q}(\Omega)$.

\medbreak

{\it Step 3} -- $w\in \widetilde \Ncal_0$.  We need to show that 
$$\lambda \int_\Omega w_1^+ Kw_2\, dx+\mu \int_\Omega w_1 Kw_2^+\, dx >0$$
and 
$$\lambda \int_\Omega w_1^- Kw_2\, dx+\mu \int_\Omega w_1 Kw_2^-\, dx <0.$$
From Step 2 and the continuity of $K$, we infer that the right-hand side in \eqref{PS2} and \eqref{PS3} do converge. We now show that the left-hand side in \eqref{PS2} and \eqref{PS3} are bounded away from zero.
Starting from
\eqref{PS2}, we get
\begin{eqnarray}\label{PS6}
\lambda a_n^+ +\mu b_n^+&\leq&  \lambda \int_{\Omega}w_{1,n}^{+}Kw_{2,n}^+\, dx+\mu \int_{\Omega}w_{2,n}^{+}Kw_{1,n}^+\, dx = 2 \int_{\Omega}w_{1,n}^{+}Kw_{2,n}^+,
\end{eqnarray}
which by \eqref{eq:maj_for_Tw} yields, for any $\delta>0$,
\begin{align}\label{PS7}
\lambda a_n^+ +\mu b_n^+&\leq c \|w_{1,n}\|_{\frac{p+1}{p},\frac{\alpha}{p}} \|w_{2,n}\|_{\frac{q+1}{q},\frac{\beta}{q}} \leq \delta a_n^+ +\frac{C}{\delta}(b_n^+)^{\frac{q(p+1)}{q+1}}.
\end{align}
Since $q(p+1)/(q+1)>1$, we deduce that $b_n^+\geq \bar \delta>0$ for some $\bar \delta>0$. The inequalities
\begin{equation}\label{PS8}
 b_n^{-}\geq \bar \delta>0, \qquad a_n^{\pm}\geq \bar \delta>0
\end{equation}
follow by arguing in a similar way.

\medbreak 

{\it Conclusion} -- By Lemma \ref{lemma:uniquemax_eps}, we can take $(t_0,s_0)$ such that 
$$(t_0^{\lambda}w_1^+-s_0^{\lambda}w_1^-,t_0^{\mu}w_2^+-s_0^{\mu}w_2^-)\in \Ncal^\eps_{\rm nod}.$$ 
By the uniqueness assertion in the same lemma and the weak lower semicontinuity of the norm, we infer that
\begin{align*}
\tilde c^\eps_\text{nod} &\leq I_\eps(t_0^{\lambda}w_1^+-s_0^{\lambda}w_1^-,t_0^{\mu}w_2^+-s_0^{\mu}w_2^-)\\
		&\leq \liminf I_\eps(t_0^{\lambda}w_{1,n}^+-s_0^{\lambda}w_{1,n}^-,t_0^{\mu}w_{2,n}^+-s_0^{\mu}w_{2,n}^-)\leq \liminf I_\eps(w_n)=\tilde c^\eps_\text{nod}.
\end{align*}
Hence $(t_0^{\lambda}w_1^+-s_0^{\lambda}w_1^-,t_0^{\mu}w_2^+-s_0^{\mu}w_2^-)\in \Ncal^\eps_{\rm nod}$ achieves $\tilde c^\eps_{\rm nod}$. 

\medbreak

At last, the characterization \eqref{eq:levelseps} of the critical level can be proved in a straightforward way.
\end{proof}

Our strategy to prove Theorem \ref{thm:levelsequiv} now essentially consists in passing to the limit in \eqref{eq:perturbed_system_eps} when $\varepsilon\to0$. As a first step, we prove the convergence of the critical level, namely $\tilde c^\eps_{\rm nod}\to \tilde c_{\rm nod}$ as $\eps\to 0$. We start with two preliminary lemmas.

\begin{lemma}\label{lemma:proj_continuity1}
Take $(w_1,w_2)\in \widetilde \Ncal_0\subset \Ncal_0$ and let $(t_0,s_0)$ be the unique pair such that
\[
t_0,s_0>0,\qquad (t_0^\lambda w_1^+-s_0^\lambda w_1^-,t_0^\mu w_2^+-s_0^\mu w_2^-)\in \Ncal_{\rm nod}
\] 
while, for each $\eps>0$, let $(t_\eps,s_\eps)$ be the unique pair such that
\[
t_\eps,s_\eps>0,\qquad (t_\eps^\lambda w_1^+-s_\eps^\lambda w_1^-,t_\eps^\mu w_2^+-s_\eps^\mu w_2^-)\in \Ncal_{\rm nod}^\eps.
\]
Then 
\[
(t_\eps,s_\eps)\to (t_0,s_0)\qquad \text{ as } \eps\to 0.
\]
\end{lemma}
\begin{proof} 
The pair $(t_\eps,s_\eps)$ solves \eqref{eq:system_t_s_eps}, so that
\[
2 B^+ t_\eps^{2-\gamma}\geq \gamma A^+\ \text{ and }\ 2B^-  s_\eps^{\gamma-2} \geq \gamma A^-.
\]
Since $2-\gamma>0$, we infer that $t_\eps,s_\eps\geq a>0$ for some constant $a$ independent of $\eps$. 

\medbreak

For the sake of contradiction, assume that $\{(t_\eps,s_\eps)\}$ is unbounded as $\eps\to 0$. 

\medbreak
\noindent {\it Case 1} --  there exists $b>0$ such that $a\leq s_\eps\leq b$ or $a\leq t_\eps\leq b$ for $\eps\in(0,1]$. In the first alternative, taking the limit in the first equation of \eqref{eq:system_t_s_eps}, we obtain $2B^+=0$ which is a contradiction. In the second alternative, taking the limit in the second equation of \eqref{eq:system_t_s_eps} leads to $2B^-=0$ which is still a contradiction. 

\medbreak
\noindent {\it Case 2} -- both $t_\eps,s_\eps\to +\infty$ as $\eps\to 0$. We divide this case in two subcases. 
\medbreak
\noindent {\it Case 2.1} -- $t_\eps/s_\eps\to +\infty$ or $t_\eps/s_\eps\to 0$. This case leads again to either $2B^+=0$ or $2B^-=0$.
\medbreak
\noindent {\it Case 2.2} -- $t_\eps/s_\eps\to l\in \R^+$. Taking the limit in \eqref{eq:system_t_s_eps} gives
\[\left\{
\begin{array}{l}
2B^+=\lambda C_1(1/l)^\mu+\mu C_2(1/l)^\lambda\\
2B^-=\mu C_1l^\lambda+\lambda C_2 l^\mu.
\end{array}\right.
\]
If $l\leq 1$, then $2B^-\leq \mu C_1+\lambda C_2$ whereas $2B^+<\lambda C_1+\mu C_2$ if $l>1$. In both cases, we obtain an inequality which contradicts the fact that $(w_1,w_2)\in \widetilde \Ncal_0$.

\medbreak

We now conclude that, up to a subsequence, $t_\eps\to \bar t>0$, $s_\eps\to \bar s>0$, which satisfy \eqref{eq:system_t_s}. Hence, the uniqueness assertion in Proposition \ref{prop:unique_maximum} implies $(\bar t,\bar s)=(t_0,s_0)$.
\end{proof}

We just proved the continuity of the projection on $\Ncal_{\rm nod}^\eps$ when $\varepsilon\to 0$. We will need also the continuity of the projection on $\Ncal_{\rm nod}$ with respect to strong convergence in $X$.

\begin{lemma}\label{lemma:proj_continuity2}
Take $(w_1,w_2)\in \Ncal_0$ and $(w_{1,n},w_{2,n})\in \widetilde \Ncal_0$ such that
\[
(w_{1,n},w_{2,n})\to (w_1,w_2) \qquad \text{ in } X,\text{ as } n\to \infty.
\]
Let $(t_n,s_n)$ and $(t_0,s_0)$ be the unique pairs of positive components such that
\[
(t_n^\lambda w_{1,n}^+-s_n^\lambda w_{1,n}^-,t_n^\mu w_{2,n}^+-s_n^\mu w_{2,n}^-),(t_0^\lambda w_{1}^+-s_0^\lambda w_{1}^-,t_0^\mu w_{2}^+-s_0^\mu w_{2}^-)\in \Ncal_{\rm nod}
\] 
Then 
\[
(t_n,s_n)\to (t_0,s_0)\qquad \text{ as } n\to \infty.
\]
\end{lemma}
\begin{proof}
We have
\[
\begin{cases}
2B_n^+=\gamma A_n^+ t_n^{\gamma-2}+\lambda C_{1,n} (s_n/t_n)^\mu + \mu C_{2,n} (s_n/t_n)^\lambda, \\
2B_n^-=\gamma A_n^- s_n^{\gamma-2}+\mu C_{1,n} (t_n/s_n)^{\lambda} + \lambda C_{2,n} (t_n/s_n)^{\mu}.
\end{cases}
\]
From the strong convergence in $X$ and the continuity of $K$, we deduce that
\begin{multline}
A_n^\pm:= \frac{p}{p+1}\int_\Omega |w_{1,n}^\pm|^\frac{p+1}{p}|x|^{-\frac{\alpha}{p}}\, dx +\frac{q}{q+1}\int_\Omega |w_{2,n}^\pm|^\frac{q+1}{q}|x|^{-\frac{\beta}{\alpha}}\, dx\\
		\to \frac{p}{p+1}\int_\Omega |w_1^\pm|^\frac{p+1}{p} |x|^{-\frac{\alpha}{p}} \, dx +\frac{q}{q+1}\int_\Omega |w_2^\pm|^\frac{q+1}{q} |x|^{-\frac{\beta}{\alpha}}\, dx=:A^\pm>0,
\end{multline}
(recall that $w_i^\pm \not \equiv 0$, $i=1,2$ whenever $w\in \Ncal_0$),
\[
B_n^\pm:= \int_\Omega w_{1,n}^\pm K w_{2,n}^\pm\, dx\to \int_\Omega w_{1,n}^\pm K w_{2,n}^\pm\, dx =: B^\pm>0
\]
and
\[
C_{1,n}:=\int_\Omega w_{1,n}^+ K w_{2,n}^-\, dx\to \int_\Omega w_{1}^+ K w_{2}^-\, dx=:C_1>0, 
\]
\[
C_{2,n}:=\int_\Omega w_{2,n}^+ K w_{1,n}^-\, dx\to \int_\Omega w_{2}^+ K w_{1}^-\, dx=:C_2>0.
\]
Using Proposition \ref{prop:unique_maximum} and arguing exactly as in the proof of Lemma \ref{lemma:proj_continuity1}, we can infer that $(t_n,s_n)_{n}$ is a bounded sequence which actually converges to $(t_0,s_0)$.
\end{proof}

We can now turn to the convergence of the critical level which implies that the extension of the map $\R^{+}\owns\eps\mapsto\tilde c_{\rm nod}^\eps$ by $\tilde c_{\rm nod}^0=\tilde c_{\rm nod}$ is right-continuous at zero. 
 
\begin{proposition}\label{prop:cont-level}
We have $\tilde c_{\rm nod}^\eps\to \tilde c_{\rm nod}$, as $\eps\to 0$.
\end{proposition}
\begin{proof}

We deal successively with the upper and lower-semicontinuity. 

\medbreak

{\it Step 1 }-- Upper semi-continuity. 

\medbreak

Fix $w\in \Ncal_0$. Since $\widetilde X$ is dense in $X$ and $\Ncal_0$ is open, there exists $(w_n)_{n}\subset \widetilde \Ncal_0$ such that $w_n\to w$ strongly in $X$. Given $\eps>0$ and $n\in \N$, according to Lemma \ref{lemma:uniquemax_eps}, there exist unique $t_{n,\eps},s_{n,\eps}>0$ such that
\[
(t_{n,\eps}^\lambda w_{1,n}^+-s_{n,\eps}^\lambda w_{1,n}^-,t_{n,\eps}^\mu w_{2,n}^+-s_{n,\eps}^\mu w_{2,n}^-)\in \Ncal_{\rm nod}^\eps.
\]
Therefore, we have
\begin{align}\label{eq:upperbound_c_eps}
\tilde c^\eps_{\rm nod} \leq& I_\eps(t_{n,\eps}^\lambda w_{1,n}^+-s_{n,\eps}^\lambda w_{1,n}^-,t_{n,\eps}^\mu w_{2,n}^+-s_{n,\eps}^\mu w_{2,n}^-)\\
				=& I(t_{n,\eps}^\lambda w_{1,n}^+-s_{n,\eps}^\lambda w_{1,n}^-,t_{n,\eps}^\mu w_{2,n}^+-s_{n,\eps}^\mu w_{2,n}^-)\\
					&+\eps t_{n,\eps}^\gamma \left(\frac{p}{p+1} \int_\Omega |\nabla w_{1,n}^+|^\frac{p+1}{p}+\frac{q}{q+1}\int_\Omega |\nabla w_{2,n}^+|^\frac{q+1}{q}\right)\\
					&+\eps s_{n,\eps}^\gamma \left(\frac{p}{p+1} \int_\Omega |\nabla w_{1,n}^-|^\frac{p+1}{p}+\frac{q}{q+1}\int_\Omega |\nabla w_{2,n}^-|^\frac{q+1}{q}\right).
\end{align}
Now observe that by Lemma \ref{lemma:proj_continuity1}, for each fixed $n\in \N$, we have
\[
(t_{n,\eps},s_{n,\eps})\to (t_{n,0},s_{n,0}) \qquad \text{ as } \eps\to 0,
\]
where $t_{n,0},s_{n,0}>0$ is the unique pair of positive components such that
\[
(t_{n,0}^\lambda w_{1,n}^+-s_{n,0}^\lambda w_{1,n}^-,t_{n,0}^\mu w_{2,n}^+-s_{n,0}^\mu w_{2,n}^-)\in \Ncal_{\rm nod}.
\]
Taking the limit in \eqref{eq:upperbound_c_eps} as $\eps\to 0$, we obtain
\[
\limsup_{\eps\to 0} \tilde c^\eps_{\rm nod}\leq I(t_{n,0}^\lambda w_{1,n}^+-s_{n,0}^\lambda w_{1,n}^-,t_{n,0}^\mu w_{2,n}^+-s_{n,0}^\mu w_{2,n}^-).
\]
On the other hand, by Lemma \ref{lemma:proj_continuity2}, we have
\[
(t_{n,0},s_{n,0})\to (t_{0},s_{0}) \qquad \text{ as } n \to \infty,
\]
where $t_{0},s_{0}>0$ is the unique pair of positive components such that
\[
(t_{0}^\lambda w_{1}^+-s_{0}^\lambda w_{1}^-,t_{0}^\mu w_{2}^+-s_{0}^\mu w_{2}^-)\in \Ncal_{\rm nod}.
\]
Hence we deduce that
\begin{align*}
\limsup_{\eps\to 0} \tilde c_{\rm nod}^\eps &\leq I(t_{0}^\lambda w_{1}^+-s_{0}^\lambda w_{1}^-,t_{0}^\mu w_{2}^+-s_{0}^\mu w_{2}^-)\\
								&=\sup_{t,s>0} I(t^\lambda w_{1}^+-s^\lambda w_{1}^-,t^\mu w_{2}^+-s^\mu w_{2}^-).
\end{align*}
Since this holds for every $w\in \Ncal_0$, \eqref{eq:goal_easyequality} implies
\[
\limsup_{\eps\to 0} \tilde c_{\rm nod}^\eps \leq \inf_{w\in \Ncal_0} \sup_{t,s>0} I(t^\lambda w_{1}^+-s^\lambda w_{1}^-,t^\mu w_{2}^+-s^\mu w_{2}^-)=\tilde c_{\rm nod}.
\]

\medbreak

{\it Step 2 }-- Lower semi-continuity. 

\medbreak Take $w_\eps\in \Ncal_{\rm nod}^\eps$ such that $I_\eps(w_\eps)=\tilde c_{\rm nod}^\eps$ and $I_\eps'(w_\eps)=0$. Since $\Ncal_{\rm nod}^\eps\subset \widetilde \Ncal_0\subset \Ncal_0$, there exists unique $(t_\eps,s_\eps)\in \R^+\times \R^+$ such that
\[
(t_\eps^\lambda w_{1,\eps}^+-s_\eps^\lambda w_{1,\eps}^-,t_\eps^\mu w_{2,\eps}^+-s_\eps^\mu w_{2,\eps}^-)\in \Ncal_{\rm nod}
\]
Therefore, we have
\begin{align}\label{eq:upperbound_c_nod}
\tilde c_{\rm nod}\leq & \,  I(t_\eps^\lambda w_{1,\eps}^+-s_\eps^\lambda w_{2,\eps}^-, t_\eps^\mu w_{2,\eps}^+-s_{\eps}^\mu w_{2,\eps}^-)\\
			\leq & \, I(t_\eps^\lambda w_{1,\eps}^+-s_\eps^\lambda w_{2,\eps}^-, t_\eps^\mu w_{2,\eps}^+-s_{\eps}^\mu w_{2,\eps}^-)\\
			& \, + \eps t_\eps^\gamma \left( \frac{p}{p+1}\int_\Omega |\nabla w_{1,\eps}^+|^\frac{p+1}{p} + \frac{q}{q+1}\int_\Omega |\nabla w_{2,\eps}^+|^\frac{q+1}{q}  \right)\\
			& \, + \eps s_\eps^\gamma \left( \frac{p}{p+1}\int_\Omega |\nabla w_{1,\eps}^-|^\frac{p+1}{p} + \frac{q}{q+1}\int_\Omega |\nabla w_{2,\eps}^-|^\frac{q+1}{q}  \right)\\
			=& \, I_\eps(t_\eps^\lambda w_{1,\eps}^+-s_\eps^\lambda w_{2,\eps}^-, t_\eps^\mu w_{2,\eps}^+-s_{\eps}^\mu w_{2,\eps}^-)\\
			\leq & \, \sup_{t,s>0} I_\eps (t^\lambda w_{1,\eps}^+-s^\lambda w_{2,\eps}^-, t^\mu w_{2,\eps}^+-s^\mu w_{2,\eps}^-) = I_\eps(w_\eps)=\tilde c_{\rm nod}^\eps.\qedhere
\end{align}
\end{proof}

Consider now the family of approximating minimizers
$$W_{\eps}:=\{(w_\eps,\eps)\in \Ncal^\eps_{\rm nod}\times \R^{+}\mid I_{\varepsilon}(w_{1,\eps},w_{2,\eps})=\tilde c_{\rm nod}^\eps\}.$$

Our subsequent step is to prove that, given a sequence $(\eps_{n})_{n}$ converging to zero and a sequence $(w_{\eps_{n}})_n$ such that $(w_{\eps_{n}},\eps_{n})\in W_{\eps_{n}}$,  $(w_{\eps_{n}})_{n}$ converges strongly in $X$ and achieves $\tilde c_{\rm nod}$. 
 
Arguing exactly as in the proof of Proposition \ref{prop:levelsequiv_eps}, we infer that
\begin{multline}\label{eq:equivalent_expression_I_eps(w_eps)}
I_\eps(w_\eps)=\frac{(pq-1)p}{(p+1)(2pq+p+q)} \left(\int_\Omega  |w_{1,\eps}|^\frac{p+1}{p}|x|^{-\frac{\alpha}{p}}\, dx+ \eps\int_\Omega |\nabla w_{1,\eps}|^\frac{p+1}{p}\right)\\
				 +\frac{(pq-1)q}{(q+1)(2pq+p+q)}  \left(\int_\Omega  |w_{2,\eps}|^\frac{1+1}{1}|x|^{-\frac{\beta}{q}}\, dx + \eps\int_\Omega |\nabla w_{2,\eps}|^\frac{q+1}{q}\right).
\end{multline}
Combining this identity with Proposition \ref{prop:cont-level}, we deduce the existence of $C>0$ such that
\begin{equation}\label{eq:upperbounds_eps}
\sup_{\eps\in(0,1]}\left(\int_\Omega  |w_{1,\eps}|^\frac{p+1}{p}|x|^{-\frac{\alpha}{p}}\, dx, \int_\Omega  |w_{2,\eps}|^\frac{q+1}{q}|x|^{-\frac{\beta}{q}}\right)\leq C.
\end{equation}
Moreover, arguing as in Step 2 of Proposition \ref{prop:levelsequiv_eps}, we deduce that  
\begin{equation}\label{eq:lowerbounds_eps}
\inf_{\eps\in(0,1]}\left(\int_\Omega | w_{1,\eps}^\pm|^{\frac{p+1}{p}}|x|^{-\frac{\alpha}{p}}\, dx,\int_\Omega | w_{2,\eps}^\pm|^{\frac{q+1}{q}}|x|^{-\frac{\beta}{q}}\, dx\right) \geq \bar\delta>0,
\end{equation}
together with the lower estimates
\begin{equation}\label{eq:uniform_eps_in_N_01}
\lambda \int_{\Omega}w_{1,\eps}^{+}Kw_{2,\eps}\, dx+\mu \int_{\Omega}w_{2,\eps}^{+}Kw_{1,\eps} \, dx\geq 2\bar \delta,
\end{equation}
\begin{equation}\label{eq:uniform_eps_in_N_02}
-\lambda \int_{\Omega}w_{1,\eps}^{-}Kw_{2,\eps}\, dx-\mu \int_{\Omega}w_{2,\eps}^{-}Kw_{1,\eps}\, dx\geq 2\bar \delta,
\end{equation}
which hold for every $\eps\in(0,1]$.

Next we prove that the gradient terms disappear when taking the limit in $I'_{\eps}(w_{\eps})$ as $\eps\to 0$. 
\begin{proposition}\label{eq:convergence_of_grad_as_epsto0}
Let $(w_\eps,\eps)\in W_{\eps}$. We have 
\[
\max\left(\eps \int_\Omega |\nabla w_{1,\eps}|^\frac{p+1}{p},\quad \eps \int_\Omega |\nabla w_{2,\eps}|^\frac{q+1}{q}\right)\to 0,\qquad \text{ as } \eps\to 0.
\]
\end{proposition}
\begin{proof}
From the inequalities in \eqref{eq:upperbound_c_nod}, we actually deduce that
\[
\lim_{\eps\to 0}I(t_\eps^\lambda w_{1,\eps}^+-s_\eps^\lambda w_{2,\eps}^-, t_\eps^\mu w_{2,\eps}^+-s_{\eps}^\mu w_{2,\eps}^-)=\tilde c_{\rm nod},
\]
\[
\lim_{\eps\to 0} \eps t_\eps^\gamma \left( \frac{p}{p+1}\int_\Omega |\nabla w_{1,\eps}^+|^\frac{p+1}{p} + \frac{q}{q+1}\int_\Omega |\nabla w_{2,\eps}^+|^\frac{q+1}{q}\right)=0,
\]
and
\[
\lim_{\eps\to 0} \eps s_\eps^\gamma \left( \frac{p}{p+1}\int_\Omega |\nabla w_{1,\eps}^-|^\frac{p+1}{p} + \frac{q}{q+1}\int_\Omega |\nabla w_{2,\eps}^-|^\frac{q+1}{q} \right) =0.
\]
The conclusion is an obvious consequence of the following claim.
\medbreak

{\it Claim} --
$t_\eps,s_\eps\not \to 0$. 
We argue as in the proof of Lemmas \ref{lemma:proj_continuity1} and \ref{lemma:proj_continuity2}, though we proceed with extra care since we have no information yet about the convergence of $(w_\eps)$ in $X$. The pair $(t_\eps,s_\eps)$ satisfies 
\begin{equation}\label{eq:system_teps_s_eps}
2t_\eps^{2-\gamma}B_\eps^+\geq \gamma A_\eps^+\qquad 2 s_\eps^{2-\gamma }B_\eps^-\geq \gamma A_\eps^- 
\end{equation}
with
\[
A_\eps^\pm:= \frac{p}{p+1}\int_\Omega |w_{1,\eps}^\pm|^\frac{p+1}{p}|x|^{-\frac{\alpha}{p}}\, dx +\frac{q}{q+1}\int_\Omega |w_{2,\eps}^\pm|^\frac{q+1}{q}|x|^{-\frac{\beta}{\alpha}}\, dx
\]
and
\[
B_\eps^\pm:= \int_\Omega w_{1,\eps}^\pm K w_{2,\eps}^\pm\, dx.
\]
From \eqref{eq:upperbounds_eps}, we deduce the existence of $w, (f_+,g_+), (f_-,g_-)  \in X$ such that, up to a subsequence, $w_\eps\stackrel{X}{\rightharpoonup} w$,
\[
w_{1,\eps}^\pm {\rightharpoonup} f_{\pm}\geq 0, \text{ weakly in } L^{\frac{p+1}{p}}(\Omega, |x|^{-\frac{\alpha}{p}})
\]
and 
\[
w_{2,\eps}^\pm \rightharpoonup g_{\pm}\geq 0, \text{ weakly in } L^{\frac{q+1}{q}}(\Omega, |x|^{-\frac{\beta}{q}}).
\]
Taking the limit as $\eps\to 0$ in \eqref{eq:uniform_eps_in_N_01}--\eqref{eq:uniform_eps_in_N_02}, we obtain
\[
\lambda \int_{\Omega}f_+\, K(g_+-g_-)\, dx+\mu \int_{\Omega}g_+\, K(f_+-f_-) \, dx\geq 2\bar \delta,
\]
\[
-\lambda \int_{\Omega}f_-\,K(g_+-g_-)\, dx-\mu \int_{\Omega}g_-K (f_+-f_-)\, dx\geq 2\bar \delta,
\]
and it is clear that $f_\pm,g_\pm\not\equiv 0$. Going back to \eqref{eq:system_teps_s_eps}, we see that the lower estimate \eqref{eq:lowerbounds_eps} yields $A_\eps^\pm \geq \tilde \delta>0$, where $\tilde \delta$ is independent of $\eps$. Since moreover
\[
B_\eps^\pm\to \int_\Omega f_{\pm}\, K g_\pm\, dx >0,
\]
we have that $t_\eps,s_\eps\not\to 0$, as claimed.
\end{proof}

\medbreak

We are now ready to conclude this section by proving its main result.

\begin{proof}[{Proof of Theorem \ref{thm:levelsequiv}}]
The key point is the next claim.

\medbreak

{\it Claim} --  up to a subsequence, $w_\eps\to w$ strongly in $X$ for some $w\in X$.

\medbreak 

We first deduce from \eqref{eq:upperbounds_eps} the existence of $w\in X$, $g_1\in L^{p+1}(\Omega,|x|^{-\frac{\alpha}{p}})$, $g_2\in L^{q+1}(\Omega,|x|^{-\frac{\beta}{q}})$ such that $w_\eps \rightharpoonup w$ weakly in $X$, 
\[
|w_{1,\eps}|^{\frac{1}{p}-1}w_{1,\eps} \rightharpoonup g_1 \text{ weakly in } L^{p+1}(\Omega, |x|^{-\frac{\alpha}{p}})	
\]
and
\[ |w_{2,\eps}|^{\frac{1}{q}-1}w_{2,\eps} \rightharpoonup g_2 \text{ weakly in } L^{q+1}(\Omega, |x|^{-\frac{\beta}{q}}).
\]
In particular, taking $\eps\to 0$ in the approximating system \eqref{eq:perturbed_system_eps} and using Proposition \ref{eq:convergence_of_grad_as_epsto0}, we conclude that
\begin{equation}\label{eq:alternative_system}
g_1 |x|^{-\frac{\alpha}{p}}=K w_2,\qquad g_2 |x|^{-\frac{\beta}{q}}=K w_1.
\end{equation}
This implies in particular that
\[
g_1 |x|^{-\frac{\alpha}{p}}\in W^{2,\frac{q+1}{q}}(\Omega)\cap W_0^{1,\frac{q+1}{q}}(\Omega)\subset L^{p+1} (\Omega)
\]
and
\[
 g_2 |x|^{-\frac{\beta}{q}}\in W^{2,\frac{p+1}{q}}(\Omega)\cap W_0^{1,\frac{p+1}{q}}(\Omega)\subset L^{q+1}(\Omega).
\]
Writing $g_1=|f_1|^{\frac{1}{p}-1}f_1$, with $f_1=|g_1|^{p-1}g_1$, we observe that $f_1\in L^\frac{p+1}{p}(\Omega,|x|^{-\frac{\alpha}{q}})$ because
\[
\int_\Omega |f_1|^\frac{p+1}{p} |x|^{-\frac{\alpha}{p}}\, dx = \int_\Omega |g_1|^{p+1} |x|^{-\frac{\alpha}{p}}\, dx <\infty.
\]
Now take $\delta>0$ and fix $h_\delta\in W^{1,\frac{p+1}{p}}(\Omega)\cap L^\frac{p+1}{p}(\Omega,|x|^{-\frac{\alpha}{p}})$ such that 
\begin{equation}\label{eq:approximation_delta}
\|f_1-h_\delta\|_{\frac{p+1}{p},\frac{\alpha}{p}}<\delta.
\end{equation}
Using the test function $w_{1,\eps}-f_1$ in the equation 
\[
-\eps \text{div}(|\nabla w_{1,\eps}|^{\frac{1}{p}-1}\nabla w_{1,\eps})+ (|w_{1,\eps}|^{\frac{1}{p}-1}w_{1,\eps}-|f_1|^{\frac{1}{p}-1}f_1)|x|^{-\frac{\alpha}{p}}=Kw_{2,\eps}-Kw_2,
\]
we deduce that
\begin{multline*}
\int_\Omega (|w_{1,\eps}|^{\frac{1}{p}-1}w_{1,\eps}-|f_1|^{\frac{1}{p}-1}f_1) (w_{1,\eps}-h_\delta)|x|^{-\frac{\alpha}{p}} \, dx\\
=-\eps\int_\Omega |\nabla w_{1,\eps}|^{\frac{1}{p}-1}\nabla w_{1,\eps}\cdot \nabla(w_{1,\eps}-h_\delta)\, dx+\int_\Omega (Kw_{2,\eps}-Kw_2) (w_{1,\eps}-h_\delta)\, dx.
\end{multline*}
Proposition \ref{eq:convergence_of_grad_as_epsto0}, the upper bounds \eqref{eq:upperbounds_eps} and the compactness of the operator $K$ now imply that the right hand side of the last equation converges to 0 as $\eps\to 0$. Consequently, we can take $\bar \eps>0$ such that
\begin{equation}\label{eq:lasteq}
\left| \int_\Omega (|w_{1,\eps}|^{\frac{1}{p}-1}w_{1,\eps}-|f_1|^{\frac{1}{p}-1}f_1) (w_{1,\eps}-h_\delta)|x|^{-\frac{\alpha}{p}} \, dx \right| <\delta, \text{ for every }\eps<\bar \eps.
\end{equation}
Combining \eqref{eq:upperbounds_eps}, \eqref{eq:approximation_delta} and  \eqref{eq:lasteq}, we infer that
\begin{multline*}
\left| \int_\Omega (|w_{1,\eps}|^{\frac{1}{p}-1}w_{1,\eps}-|f_1|^{\frac{1}{p}-1}f_1) (w_{1,\eps}-f_1)|x|^{-\frac{\alpha}{p}} \, dx \right|\\
\leq \left| \int_\Omega (|w_{1,\eps}|^{\frac{1}{p}-1}w_{1,\eps}-|f_1|^{\frac{1}{p}-1}f_1) (w_{1,\eps}-h_\delta)|x|^{-\frac{\alpha}{p}} \, dx \right| \\
+\left| \int_\Omega (|w_{1,\eps}|^{\frac{1}{p}-1}w_{1,\eps}-|f_1|^{\frac{1}{p}-1}f_1) (h_\delta-f_1)|x|^{-\frac{\alpha}{p}} \, dx \right| \\ 
\leq \delta + \|h_\delta-f_1\|_{\frac{p+1}{p},\frac{\alpha}{p}} \| |w_{1,\eps}|^{\frac{1}{p}-1}w_{1,\eps}-|f_1|^{\frac{1}{p}-1}f_1 \|_{p+1,\frac{\alpha}{p}}\leq \delta+C \delta,
\end{multline*}
for some $C>0$ (independent of $\eps$ and $\delta$). We therefore conclude that
\begin{equation}\label{eq:happy_end}
\lim_{\eps \to 0} \int_\Omega (|w_{1,\eps}|^{\frac{1}{p}-1}w_{1,\eps}-|f_1|^{\frac{1}{p}-1}f_1) (w_{1,\eps}-f_1)|x|^{-\frac{\alpha}{p}} \, dx =0.
 \end{equation} 
Using the classical pointwise estimate 
\[
(|\xi|^{\frac{1}{p}-1}\xi-|\eta|^{\frac{1}{p}-1}\eta)\cdot (\xi-\eta)\geq 2^\frac{p-1}{p} |\xi-\eta|^\frac{p+1}{p} \qquad \text{ if } 0<p<1,
\]
see for instance \cite{Simon},
we easily deduce from \eqref{eq:happy_end} that $w_{1,\eps}\to f_1=w_1$ strongly in $L^{\frac{p+1}{p}}(\Omega, |x|^{-\frac{\alpha}{p}})$ when $p\leq1$. In the complementary case $p>1$, using the pointwise estimate 
\[
(|\xi|^{\frac{1}{p}-1}\xi-|\eta|^{\frac{1}{p}-1}\eta)\cdot (\xi-\eta)\geq \frac{1}{p}|\xi-\eta|^2(|\xi|+|\eta|)^{\frac{1}{p}-1} \qquad \text{ if } p\geq 1,
\]
see again \cite{Simon}, we observe that
\begin{multline*}
\int_\Omega |w_{1,\eps}-f_1|^\frac{p+1}{p} |x|^{-\frac{\alpha}{p}}\, dx =\int_\Omega \frac{|w_{1,\eps}-f_1|^\frac{p+1}{p}}{(|w_{1,\eps}|+|f_1|)^\frac{p^2-1}{2p^2}}(|w_{1,\eps}|+|f_1|)^\frac{p^2-1}{2p^2} |x|^{-\frac{\alpha}{p}}\, dx\\
\leq \left( \int_\Omega \frac{|w_{1,\eps}-f_1|^2}{(|w_{1,\eps}+|f_1|)^{1-\frac{1}{p}}} |x|^{-\frac{\alpha}{p}} \, dx\right)^\frac{p+1}{2p} \left(\int_\Omega (|w_{1,\eps}|+|f_1|)^\frac{p+1}{p} |x|^{-\frac{\alpha}{p}}\, dx\right)^\frac{p-1}{2p}\\
\leq C \left(\int_\Omega (|w_{1,\eps}|^{\frac{1}{p}-1}w_{1,\eps}-|f_1|^{\frac{1}{p}-1}f_1) (w_{1,\eps}-f_1)|x|^{-\frac{\alpha}{p}} \, dx\right)^\frac{p+1}{2p}
\end{multline*}
and we reach the same conclusion as for $p\le 1$. 
Obviously, the convergence of the component $w_{2,\eps}$ follows in an analogous way.

\medbreak

The previous claim guarantees that $w_\eps\to w$ with $w_1^\pm, w_2^\pm\not\equiv 0$ since \eqref{eq:lowerbounds_eps} now 
implies 
$$
\min\left(\int_\Omega | w_{1}^\pm|^{\frac{p+1}{p}}|x|^{-\frac{\alpha}{p}}\, dx,\int_\Omega | w_{2}^\pm|^{\frac{q+1}{q}}|x|^{-\frac{\beta}{q}}\, dx\right) \geq \bar\delta>0.
$$
From Propositions \ref{prop:cont-level} and \ref{eq:convergence_of_grad_as_epsto0} together with the strong convergence in $X$, we conclude that  $I(w)=\tilde c_{\rm nod}$. Moreover, the equations in \eqref{eq:alternative_system} tells that $w$ is actually a critical point of $I$ so that we have indeed proved that $\tilde c_{\rm nod}$ is achieved by a critical point of the functional $I$.

\medbreak

At last, the characterization \eqref{align:happy_end} of the critical level follows in a straightforward way as previously mentioned.
\end{proof}

\section{Least energy nodal solutions are foliated Schwarz symmetric}\label{sec:SchwarzFoliated}

Let $\Omega$ be a bounded radial domain centred at the origin, namely a ball or an annulus. The purpose of this section is to prove Theorem \ref{thm:main2} via polarization methods, in the spirit of \cite{BartschWethWillem, TavaresWeth}. First, we introduce some definitions and recall some known results. Define the sets 
$$\Hcal_0=\{H\subset \R^N:\ H \text{ is a closed half-space in $\R^N$ with $0\in \partial H$}\}$$
and, for $p\neq 0$,
$$\Hcal_0(p)=\{H\in \Hcal_0:\ p\in {\rm int}(H)\}.$$
For each $H\in \Hcal_0$ we denote the reflection in $\R^N$ with respect to the hyperplane $\partial H$ by $\sigma_H:\R^N\to \R^N$, and define the polarization of a function $u:\Omega\to \R$ with respect to $H\in \Hcal_0$ by
$$
u_H(x)=
\left\{
\begin{array}{ll}
\max\{u(x), u(\sigma_H(x))\} & x\in H\cap \Omega,\\
\min\{ u(x),u(\sigma_H(x)) \} & x\in \Omega\backslash H.
\end{array}
\right.
$$

As far as we know the link between polarization and foliated Schwarz symmetry appeared first in \cite{SmetsWillem};  cf.  \cite[Theorem 2.6]{BartschWethWillem} for further results about the foliated Schwarz symmetry of least energy solutions of some second order elliptic equations with radial data. We recall from \cite[Lemma 4.2]{Brock}, see also \cite[Proposition 2.7]{Weth}, the following equivalent characterization of foliated Schwarz symmetry which involves polarization.

\begin{proposition}\label{prop: equivalent charact for Sch symmetry}
Let $u:\Omega\to \R$ be a continuous function and take $p\in \partial B_1(0)$. Then the following statements are equivalent:
\begin{enumerate}[i)]
\item $u$ is foliated Schwarz symmetric with respect to $p$; 
\item $u_H(x)=u(x)$ $\forall \, x\in \Omega\cap H$, whenever $H\in \Hcal_0(p)$.
\end{enumerate}
\end{proposition}

Moreover, the next lemma collects some known properties about polarization; cf. \cite[Lemma 3.1]{Weth} for the first property and \cite[Lemma 2.1]{BartschWethWillem} for the second.

\begin{lemma}\label{lemma: invariance properties of polarization}
Let $u:\Omega\to \R$ be a measurable function and $H\in \Hcal_0$.
\begin{enumerate}[i)]
\item If $F: \Omega \times \R\to \R$ is a continuous function
  such that $F(x,t)=F(y,t)$ for every $x,y \in \Omega$ such that
  $|x|=|y|$ and $t \in \R$ and $\displaystyle \int_\Omega |F(x,u(x))|\, dx<+\infty$, then $\displaystyle \int_\Omega F(x,u_H)\, dx=\int_\Omega F(x,u)\, dx$.
\item $(u_H)^+=(u^+)_H$, $(u_H)^-=-(-u^-)_H$.
\end{enumerate}
\end{lemma}

Observe that the second statement of the previous result implies that:
\begin{equation}\label{eq:auxiliary}
(au^+-b u^-)_H=a(u_H)^+-b(u_H)^-,\qquad \forall \, a,b>0.
\end{equation}

Finally, before we head to the proof of Theorem \ref{thm:main2}, we recall the following key estimate from \cite[Lemma 3.7]{BonheureSantosRamosJFA}.

\begin{lemma} \label{lemma: key_estimate_foliated}
Given $u\in L^{\frac{p+1}{p}}(\Omega)$, $v\in L^{\frac{q+1}{q}}(\Omega)$ and any $H\in \Hcal_0$, we have that
\[
\int_{\Omega}uKv\leqslant \int_{\Omega} u_{H}K(v_H).
\]
\end{lemma}

We are now ready to prove our second main result.

\begin{proof}[\textbf{Proof of Theorem \ref{thm:main2}}] Let $(u,v)$ be a least energy nodal solution of \eqref{eq:mainsystem} and take the corresponding pair $(w_1,w_2)\in X$. Fix any $r>0$ such that $\partial B_r(0)\subset \Omega$ and take $p\in \partial B_1(0)$ such that $w_{1}(rp)=\max_{\partial B_r(0)} w_{1}$. Given $H\in \Hcal_0(p)$, we aim at proving that $(w_1)_H(x)=w_1(x)$ and $(w_2)_H(x)=w_2(x)$ for $x\in \Omega\cap H$.
As
\begin{align*}
0&<\lambda \int_\Omega w_1^+ K w_2\, dx+\mu \int_\Omega w_1 K w_2^+\, dx\\
 &\leq \lambda \int_\Omega (w_1^+)_H K(w_2)_H\, dx+\mu \int_\Omega (w_1)_H K(w_2)^+_H\, dx
\end{align*}
and
\begin{align*}
0&<\lambda \int_\Omega (-w_1^-) K w_2\, dx+\mu \int_\Omega w_1 K (-w_2^-)\, dx\\
 &\leq \lambda \int_\Omega (-w_1^-)_H K (w_2)_H\, dx+\mu \int_\Omega (w_1)_H K (-w_2^-)_H\, dx\\
 &= -\lambda \int_\Omega ((w_1)_H)^- K (w_2)_H\, dx-\mu \int_\Omega (w_1)_H K  ((w_2)_H)^-\, dx
\end{align*}
then $((w_1)_H,(w_2)_H)\in \Ncal_0$, and from Proposition \ref{prop:unique_maximum} we know there exist $t_0,s_0>0$ such that
$(t_0^{\lambda}(w_1)_H^+-s_0^{\lambda}((w_1)_H)^-,t_0^{\mu}(w_2)_H^+ - s_0^{\mu}((w_2)_H)^-)\in \Ncal_{\rm nod}$. Thus, by putting together Lemma \ref{lemma: key_estimate_foliated} with \eqref{eq:auxiliary} and with the uniqueness of global maximum,
\begin{align*}
c_{\rm nod}&\leq I(t_0^{\lambda}(w_1)_H^+-s_0^{\lambda}((w_1)_H)^-,t_0^{\mu}(w_2)_H^+ - s_0^{\mu}((w_2)_H)^-)\\
		&=I((t_0^{\lambda}w_1^+ -s_0^{\lambda}w_1^-)_H,(t_0^{\mu}w_2^+ - s_0^{\mu}w_2^-)_H)\\
		&\leq I(t_0^{\lambda}w_1^+-s_0^{\lambda}w_1^-,t_0^{\mu}w_2^+ - s_0^{\mu}w_2^-)=\theta_w(t_0,s_0)\\
		& \leq \sup_{t,s>0} \theta_w(t,s)=\theta_w(1,1)=I(w)=c_{\rm nod}.
\end{align*}
Thus $(t_0,s_0)=(1,1)$, $((w_1)_H,(w_2)_H)\in \Ncal_{\rm nod}$ and $I((w_1)_H,(w_2)_H)=c_{\rm nod}$. By Lemma \ref{lemma:minimum_of_I_is_critical}, $I'((w_1)_H,(w_2)_H)=0$. Going bak to $(u,v)$, we have that both this pair as well as $(u_H,v_H)$ solve \eqref{eq:mainsystem}. Thus
\begin{equation}\label{eq:system_with_polarized}
-\Delta (u_H-u)=|x|^{\beta}(|v_H|^{q-1}v_H-|v|^{q-1}v),\ \ \ -\Delta (v_H-v)=|x|^{\alpha}(|u_H|^{p-1}u_H-|u|^{q-1}u),
\end{equation}
in $\Omega\cap H$, and $u_H-u=v_H-v=0$ on $\partial \Omega\cup (\Omega\cap \partial H)$. As $v_H\geq v$ in $\Omega\cap H$, then $-\Delta (u_H-u)\geq 0$ and by the maximum principle we have that either $u_H\equiv u$ or $u_H>u$. Since $u_H(rp)=u(rp)$, then $u_H\equiv u$ in $\Omega\cap H$. Going back to $\eqref{eq:system_with_polarized}$, we have $-\Delta (v_H-v)=0$, and thus also $v_H\equiv v$.
\end{proof}

\section{Symmetry breaking}\label{eq:Symmetry}

We start by proving Theorem \ref{lemma:u=v}.

\begin{proof}[\textbf{Proof of Theorem \ref{lemma:u=v}}]
Multiplying the first equation of \eqref{eq:mainsystem} by $u$, the second of \eqref{eq:mainsystem} by $v$ and integrating both gives
$$\int_\Omega |\nabla u|^2 \, dx= \int_\Omega |v|^{q-1}vu  |x|^{\beta} dx\leq \left(\int_\Omega |v|^{q+1}|x|^{\beta} dx\right)^{\frac{q}{q+1}}  \left(\int_\Omega |u|^{q+1}  |x|^{\beta} dx\right)^{\frac{1}{q+1}}$$
and 
$$\int_\Omega |\nabla v|^2  \, dx = \int_\Omega |u|^{q-1}uv |x|^{\beta} dx \leq \left(\int_\Omega |u|^{q+1}  |x|^{\beta} dx\right)^{\frac{q}{q+1}}  \left(\int_\Omega |v|^{q+1} |x|^{\beta} dx\right)^{\frac{1}{q+1}}.$$
Multiplying the first equation of \eqref{eq:mainsystem} by $v$, the second of \eqref{eq:mainsystem} by $u$ and integrating both gives
$$\int_\Omega \nabla u \cdot \nabla v  \, dx = \int_\Omega |v|^{q+1} |x|^{\beta} dx = \int_\Omega |u|^{q+1} |x|^{\beta} dx .$$
Putting these estimates together, we infer that 
$$\int_\Omega |\nabla u|^2 \, dx + \int_\Omega |\nabla v|^2 \, dx \leq 2 \int_\Omega \nabla u \cdot \nabla v  \, dx,$$
which obviously implies $u=v$.
\end{proof}
Remember that for the single equation 
\begin{equation}\label{eq:p=q}
-\Delta u = |u|^{q-1}u, \quad \mbox{ in } \Omega, \qquad u=0 \mbox{ on } \partial \Omega,
\end{equation}
it is known, cf. \cite[Theorem 1.3]{AftalionPacella}, that any least energy nodal solution is non radial when $\Omega \con \R^N$, $N \geq 2$, is either a ball or an annulus centred at the origin. We will show that when $(p,q)$ is close to some couple $(q_0,q_0)$, and $(\alpha,\beta)$ is close to $(0,0)$, this property is also true. 

Take $q_0$ satisfying
\begin{equation}\label{assumption on q_0}
q_0>1\quad \text{ such that } \quad q_0+1<2N/(N-2)\text{ if } N\geq 3,
\end{equation}
and $\delta_0$ such that
\begin{equation}\label{assumption on delta_0}
q_0-\delta_0>1 \quad \text{ and } \quad q_0+1+\delta_0<2N/(N-2)\text{ if } N\geq 3,
\end{equation}
that is, such that the square $[q_0-\delta,q_0+\delta_0]^2$ is contained in the region of the points $(p,q)$ such that \eqref{eq:subcritical} holds. 

The proof of Theorem \ref{thm:symmetrybreak_close_diag} consists in doing some asymptotic estimates of the least energy nodal solutions and levels as $p,q\to q_0$ and $\alpha,\beta\to 0$, combined with the known fact that, at the diagonal point $(q_0,q_0)$ and $\alpha=\beta=0$, least energy nodal solutions are non radial. Having this in mind, let us introduce some notations. Given $(p,q)$ satisfying \eqref{eq:subcritical}, $\alpha,\beta\geq 0$, we denote by $c_\text{nod}^{p,q,\alpha,\beta}$ the least energy nodal level of \eqref{eq:mainsystem}, and by $E_{p,q,\alpha,\beta}$ its associated energy \eqref{energy E}. We will also use the variational framework introduced in Section \ref{sec:existence}, denoting by $I_{p,q,\alpha,\beta}$ the energy functional \eqref{energy I}. Recall that $E_{p,q,\alpha,\beta}(u,v)=I_{p,q,\alpha,\beta}(w_1,w_2)$ at critical points, under the relation 
\[
(u, v):=(|x|^{-\frac{\alpha}{p}} |w_1|^{\frac{1}{p}-1}w_1, |x|^{-\frac{\beta}{q}} |w_2|^{\frac{1}{q}-1}w_2).
\]
Finally, recall the characterizations (cf. Theorem \ref{thm:levelsequiv}):
\[
c^{p,q,\a,\b}_{\rm nod} = \inf_{\Ncal_{\rm nod}^{p,q,\a,\b}} I_{p,q,\a,\b}= \inf_{w\in \Ncal_0^{p,q,\a,\b}} \sup_{t,s> 0} I_{p,q,\a,\b}(t^\lambda w_1^+-s^\lambda w_1^-,t^\mu w_2^+-s^\mu w_2^-),
\]
where 
 \[
 \lambda=\lambda(p,q):=\frac{2p(q+1)}{p+q+2pq}, \qquad \mu=\mu(p,q):=\frac{2q(p+1)}{p+q+2pq},
 \]
and
\begin{multline*}
\Ncal_{\rm nod}^{p,q,\a,\b}=\left\{(w_1, w_2)\in X_{p,q,\a,\b}: \ w_1^{\pm}\not \equiv 0, \ w_2^{\pm}\not \equiv 0 \ \text{ and }\right. \\ 
\left.I_{p,q,\a,\b}'(w)(\lambda w_1^+,\mu w_2^+)=I_{p,q,\a,\b}'(w)(\lambda w_1^-,\mu w_2^-)=0\right\},
\end{multline*}

\begin{align*}
\Ncal_0^{p,q,\a,\b}&:=\left\{w\in X_{p,q,\a,\b}:\ \begin{array}{c} \lambda \int_\Omega w_1^+ Kw_2\, dx+\mu \int_\Omega w_1 Kw_2^+\, dx >0\\[0.1cm] 
			\lambda \int_\Omega w_1^- Kw_2\, dx+\mu \int_\Omega w_1 Kw_2^-\, dx <0  \end{array}\right\},
\end{align*}
with $X_{p,q,\a,\b}=L^\frac{p+1}{p}(\Omega,|x|^{-\frac{\a}{p}})\times L^\frac{q+1}{q}(\Omega,|x|^{-\frac{\beta}{q}})$.

For simplicity, when $p=q$ and $\alpha=\beta=0$ we will use the notation $c_{\rm nod}^{p}$ for $c_{\rm nod}^{p,p,0,0}$. First we prove an uniform lower bound for the positive and negative parts of elements of $\Ncal_{\rm nod}^{p,q,\alpha,\beta}$.

\begin{lemma}\label{lemma:lower_bounds}
Given $q_0$ satisfying \eqref{assumption on q_0} there exists $\delta_0>0$ and $\eps>0$ such that
\[
\int_\Omega |w_1^\pm|^\frac{p+1}{p} |x|^{-\frac{\a}{p}}\, dx \geq \eps \ \ \text{and} \ \ \int_\Omega |w_2^\pm|^\frac{q+1}{q} |x|^{-\frac{\b}{q}}\, dx \geq \eps
\]
for every $(w_1,w_2)\in \Ncal_{\rm nod}^{p,q,\a,\b}$ with $p,q\in [q_0-\delta_0,q_0+\delta_0]$ and $\a,\b\in [0,\delta_0]$.
\end{lemma}
\begin{proof}
We use the estimates in the proof of Theorem \ref{thm:levelsequiv} - step 3, this time keeping a better track of the constants. We split the proof in several steps.

\medbreak

\noindent 1) There exists $C_1$ (independent of $p$ and $q$) such that
\[
\|u\|_{p+1}\leq C_1 \|u\|_{W^{2,\frac{q+1}{q}}} \qquad \forall u\in W^{2,\frac{q+1}{q}}(\Omega),\ p,q\in [q_0-\delta_0,q_0+\delta_0].
\]

Since $\Omega$ has finite measure and $p\leq q_0+\delta_0$, from H\"older's estimates we deduce that
\[
\|u\|_{p+1} \leq |\Omega|^\frac{q_0+\delta_0-p}{(p+1)(q_0+\delta_0+1)} \|u\|_{q_0+\delta_0+1} \leq \kappa_1 |\Omega|^\frac{q_0+\delta_0-p}{(p+1)(q_0+\delta_0+1)}  \|u\|_{W^{2,\frac{q_0+\delta_0+1}{q_0+\delta_0}}},
\]
where $\kappa_1$ is a constant associated to the embedding $W^{2,\frac{q_0+\delta_0+1}{q_0+\delta_0}} \hookrightarrow L^{q_0+\delta_0+1}$; recall that $\delta_0$ is such that \eqref{assumption on delta_0} holds. Moreover, again by using H\"older estimates and also that $(q_0+\delta_0+1)/(q_0+\delta_0)\leq (q+1)/q$,
\begin{align*}
\|u\|_{W^{2,\frac{q_0+\delta_0+1}{q_0+\delta_0}}} &= \left( \sum_{|\alpha|\leq 2} \int_\Omega |D^\alpha u|^\frac{q_0+\delta_0+1}{q_0+\delta_0}\, dx\right)^\frac{q_0+\delta_0}{q_0+\delta_0+1}\\
									& \leq \left(  |\Omega|^\frac{q_0+\delta_0-q}{(q+1)(q_0+\delta_0)} \sum_{|\alpha|\leq 2}  \|D^\alpha u\|_{\frac{q+1}{q}}^\frac{q_0+\delta_0+1}{q_0+\delta_0}   \right)^\frac{q_0+\delta_0}{q_0+\delta_0+1}\\
									&\leq \left(\frac{N(N-1)}{2}+N+1\right)^\frac{q_0+\delta_0}{q_0+\delta_0+1} |\Omega|^\frac{q_0+\delta_0-q}{(q+1)(q_0+\delta_0+1)} \|u\|_{W^{2,\frac{q+1}{q}}},
\end{align*}
and thus $\|u\|_{p+1}\leq \kappa(p,q) \|u\|_{W^{2,\frac{q+1}{q}}}$,  with 
\[ 
\kappa(p,q)=\kappa_1 \left(\frac{N(N-1)}{2}+N+1\right)^\frac{q_0+\delta_0}{q_0+\delta_0+1}  |\Omega|^\frac{q_0+\delta_0-q}{(q+1)(q_0+\delta_0+1)} |\Omega|^\frac{q_0+\delta_0-p}{(p+1)(q_0+\delta_0+1)},
\]
which is bounded from above by some $C_1$, for every $p,q\in [q_0-\delta_0,q_0+\delta_0]$.

\medbreak

\noindent 2) There exists $C_2$ such that, for all $u\in W^{2,\frac{q_0+\delta+1}{q_0+\delta}}(\Omega)\cap W^{1,\frac{q_0+\delta+1}{q_0+1}}_0(\Omega)$ 

\[
\|Ku\|_{W^{2,\frac{q_0+\delta_0+1}{q_0+\delta_0}}}\leq C_2 \|u\|_{\frac{q_0+\delta_0+1}{q_0+\delta}},
\]
cf. \cite[Lemma 9.17]{GilbargTrudinger}.
\medbreak

\noindent 3) As $(w_1,w_2)\in \Ncal_{\rm nod}^{p,q,\alpha,\beta}$, from steps 1) and 2) above there exists $C>0$ independent of $p,q,\a,\b$ such that
\begin{align*}
\lambda \int_\Omega |w_1^+|^\frac{p+1}{p} &|x|^{-\frac{\a}{p}}\, dx+ \mu \int_\Omega |w_2^+|^\frac{q+1}{q} |x|^{-\frac{\b}{q}} \, dx \leq 2\int_\Omega w_1^+ K w_2^+\, dx \\
					&\leq 2 \|w_1^+\|_{\frac{p+1}{p}}\|K w_2^+\|_{p+1} \leq 2C_1 \|w_1^+\|_{\frac{p+1}{p}} \|K w_2^+\|_{W^{2,\frac{q_0+1+\delta_0}{q_0+\delta_0}}}\\
					&\leq 2C_1C_2 \|w_1^+\|_{\frac{p+1}{p}} \|w_2^+\|_{\frac{q_0+\delta_0+1}{q_0+\delta_0}} \leq \widetilde C \|w_1^+\|_{\frac{p+1}{p}} \|w_2^+\|_{\frac{q+1}{q}} \\
					&\leq C \|w_1^+\|_{\frac{p+1}{p},\frac{\a}{p}} \|w_2^+\|_{\frac{q+1}{q},\frac{\b}{q}},
\end{align*}
where we have used estimate \eqref{eq:estimate_between_L^r} and the fact that $q\leq q_0+\delta_0$. By using the Young's inequality
\[
C ab \leq \frac{\lambda}{2} a^\frac{p+1}{p} + \frac{(2p)^{p} C^{p+1}}{\lambda^p (p+1)^{p+1}} b^{p+1} \qquad \forall a,b\geq0,
\]
we have
\[
\mu \int_\Omega |w_2^+|^\frac{q+1}{q} |x|^{-\frac{\b}{q}} \, dx \leq \frac{(2p)^{p} C^{p+1}}{\lambda^p (p+1)^{p+1}}\left(\int_\Omega |w_2^+|^\frac{q+1}{q} |x|^{-\frac{\b}{q}} \, dx\right)^\frac{q(p+1)}{q+1},
\]
and thus
\[
\int_\Omega |w_2^+|^\frac{q+1}{q} |x|^{-\frac{\b}{q}} \, dx\geq \left( \frac{\mu \lambda^p (p+1)^{p+1}}{(2p)^{p} C^{p+1}} \right)^\frac{q+1}{pq-1}=:K(p,q). 
\]
As $K(q_0,q_0)>0$, then from sufficiently small $\delta_0$ we have $K(p,q)\geq \eps>0$ for every  $p,q\in [q_0-\delta_0,q_0+\delta_0]$.

The lower bounds for the remaining integrals follow in an analogous way.
\end{proof}

\begin{lemma}\label{lemma:upperbound}
We have
\[
\limsup c_{\rm nod}^{p,q,\a,\b}\leq c_{\rm nod}^{q_0} \qquad \text{ as } p,q\to q_0,\ \a,\b\to 0.
\]
In particular, there exists $\delta_0$ and $\kappa>0$ such that
\[
0<c_{\rm nod}^{p,q,\a,\b}\leq \kappa,\qquad \forall \, p,q\in [q_0-\delta_0,q_0+\delta_0],\ \a,\b\in [0,\delta_0].
\]
\end{lemma}

\begin{proof} Take $p_n,q_n\to q_0$, $\a_n,\b_n\to 0$.

\medbreak

\noindent 1) We adapt some ideas from \cite[Lemma 3]{BonheureSchaftingen}, where a different problem is considered. Let $(w_1,w_2)$ be such that $w_1^{\pm}\not \equiv 0$, $w_2^{\pm}\not \equiv 0$,
\[
I_{q_0}(w_1,w_2)=c_{\rm nod}^{q_0},\qquad I'_{q_0}(w_1,w_2)=0,
\]
where $I_{q_0} = I_{q_0,q_0,0,0}$. Denote $\lambda_n:=\lambda(p_n,q_n)$ and $\mu_n:=\mu(p_n,q_n)$. Since $\lambda_n,\mu_n\to 1$ as $n\to \infty$ and
\[
\int_\Omega w_1^+ K w_2\, dx+\int_\Omega w_1 K w_2^+\, dx>0,\quad \int_\Omega w_1^- K w_2\, dx+\int_\omega w_1 K w_2^-\, dx<0,
\]
then $(w_1,w_2)\in \Ncal_0^{p_n,q_n,\a_n,\b_n}$ for large $n$, and
\[
c_{\rm nod}^{p_n,q_n,\a_n,\b_n}\leq \sup_{t,s>0} I_{p_n,q_n,\a_n,\b_n} (t^{\lambda_n} w_1^+ - s^{\lambda_n} w_1^-,t^{\mu_n}w_2^+-s^{\mu_n}w_2^-).
\]
Assume that the supremum at the right hand side is achieved at $(t,s)=(t_n,s_n)$.

\medbreak 

\noindent 2) We claim that $t_n,s_n\to 1$.

\smallbreak

\noindent 2a) First observe that $t_n, s_n$ are bounded. In fact, repeating the computations of Lemma \ref{lemma:global_maximum}, we have this time that
\begin{multline*}
I_{p_n,q_n,\a_n,\b_n}(t_n^{\lambda_n} w_1^+ - s_n^{\lambda_n} w_1^-,t_n^{\mu_n}w_2^+-s_n^{\mu_n}w_2^-)\leq A_n^+ t_n^{\gamma_{n}}+A_n^- s_n^{\gamma_{n}}+\\
+\left(\frac{1}{2}(\lambda_n C_1+\mu_n C_2)-B^+\right) t_n^2+\left(\frac{1}{2}(\mu_n C_1+\lambda_n C_2)-B^-\right) s_n^2,
\end{multline*}
with 
\[
A_n^\pm=\frac{p_n}{p_n+1}\int_\Omega |w_1^\pm|^\frac{p_n+1}{p_n} |x|^{-\frac{\a_n}{p_n}}\, dx +\frac{q_n}{q_n+1} \int_\Omega |w_2^\pm|^\frac{q_n+1}{q_n}|x|^{-\frac{\b_n}{q_n}}\, dx,
\] 
(positive and bounded in $n$), $B^\pm$ and $C_i$ are as in Lemma \ref{lemma:global_maximum}, and
\[
\gamma_{n}=\lambda_n\frac{p_n+1}{p_n}=\mu_n \frac{q_n+1}{q_n}\to \frac{q_0+1}{q_0}\in (1,2).
\]
Since moreover
\[
\frac{1}{2}(\lambda_n C_1+\mu_n C_2)-B^+\to \frac{1}{2}(C_1+C_2)-B^+<0,
\] 
\[
\frac{1}{2}(\mu_n C_1+\lambda_n C_2)-B^-\to \frac{1}{2}(C_1+C_2)-B^-<0,
\] 
then if $|s_n|+|t_n|\to \infty$ we would have
\[
0<c_{\rm nod}^{p_n,q_n,\a_n,\b_n}\leq I_{p_n,q_n,\a_n,\b_n}(t_n^{\lambda_n} w_1^+ - s_n^{\lambda_n} w_1^-,t_n^{\mu_n}w_2^+-s_n^{\mu_n}w_2^-)\to -\infty,
\]
a contradiction.

\smallbreak

\noindent 2b) We have $t_n,s_n\not \to 0$. In fact,
\[
(t_n^{\lambda_n} w_1^+ - s_n^{\lambda_n} w_1^-,t_n^{\mu_n}w_2^+-s_n^{\mu_n}w_2^-)\in \Ncal_{\rm nod}^{p_n,q_n,\a_n,\b_n},
\]
hence by Lemma \ref{lemma:lower_bounds}
\[
\int_\Omega |t_n^{\lambda_n} w_1^+ |^\frac{p_n+1}{p_n} |x|^{-\frac{\a_n}{p_n}}\, dx,\ \int_\Omega |s_n^{\lambda_n} w_1^- |^\frac{p_n+1}{p_n} |x|^{-\frac{\a_n}{p_n}}\, dx\geq \eps>0,
\]
which proves the statement.

\smallbreak

\noindent 2c) The claim of 2) now follows. We have $t_n\to \bar t\neq 0$, $s_n\to \bar s\neq 0$.  Since
\begin{align*}
I_{p_n,q_n,\a_n,b_n}(w_1,w_2)&\leq \sup_{t,s>0} I_{p_n,q_n,\a_n,b_n}( t^{\lambda_n} w_1^+ - s^{\lambda_n} w_1^-, t^{\mu_n} w_2^+ - s^{\mu_n} w_2^-)\\
				&= I_{p_n,q_n,\a_n,\b_n}(t_n^{\lambda_n} w_1^+ - s_n^{\lambda_n} w_1^-,t_n^{\mu_n}w_2^+-s_n^{\mu_n}w_2^-).
\end{align*}
by passing to the limit,
\[
\sup_{t,s>0} I_{q_0}( t w_1^+ - s w_1^-, t w_2^+ - s w_2^-)=I_{q_0}(w_1,w_2) \leq I_{q_0}(\bar t w_1^+ - \bar s w_1^-, \bar t w_2^+ - \bar s w_2^-).
\]
By the uniqueness provided by Proposition \ref{prop:unique_maximum}, we have $\bar t=\bar s=1$.

\medbreak

\noindent 3) Finally, by making $n\to \infty$ in the inequality
\[
c_{\rm nod}^{p_n,q_n,\a_n,\b_n}\leq I_{p_n,q_n,\a_n,\b_n}(t_n^{\lambda_n}w_1^+-s_n^{\lambda_n}w_1^-,t_n^{\mu_n}w_2^+-s_n^{\mu_n} w_2^-),
\]
we obtain
\[
\limsup c_{\rm nod}^{p_n,q_n,\a_n,\b_n}\leq I_{q_0}(w_1,w_2)=c_{\rm nod}^{q_0}.
\]
\end{proof}

As a consequence, we have the following a priori bound.
\begin{lemma}
Given $q_0$ satisfying \eqref{assumption on q_0} there exists $\delta_0>0$ and $\kappa>0$ such that
\[
\| (u,v)\|_\infty \leq \kappa  
\]
for every $(u,v)$ least energy nodal solution of \eqref{eq:mainsystem} with $p,q\in [q_0-\delta_0,q_0+\delta_0]$, $\alpha, \beta \in [0,\delta_0]$.
\end{lemma}
\begin{proof}
Having the uniform upper bound of the energy levels coming from the previous lemma, and since the nonlinearities in \eqref{eq:mainsystem} satisfy
\[
\left| |x|^\a |s|^{p_n} s\right|, \left||x|^\b |s|^{q_n}s\right|\leq C(1+|s|^{q_0+1+\delta})\quad \forall s\in \R,
\]
for $p,q\in [q_0-\delta_0,q_0+\delta_0]$, $\delta \in [0,\delta_0]$, $C$ independent of $p,q,\a,\b$, then one can reason exactly as in the proof of \cite[Lemma 5.4]{RamosTavares} (see also \cite[Theorem 5.18]{BonheureSantosTavares}) to obtain uniform $L^\infty$ bounds.
\end{proof}

\begin{lemma}\label{lemma:continuity}
Take $q_0$ satisfying \eqref{assumption on q_0}. Then
\[
c_{\rm nod}^{p,q,\a,\b}\to c_{\rm nod}^{q_0} \quad \text{ as } p,q\to q_0,\ \a,\b\to 0.
\]
Moreover, the corresponding least energy nodal solutions converge: if \linebreak $(u_{p,q,\a,\b},v_{p,q,\a,\b})$ is a sign changing solution of \eqref{eq:mainsystem} with 
\[
E_{p,q,\a,\b}(u_{p,q,\a,\b},v_{p,q,\a,\b})=c_{\rm nod}^{p,q,\a,\b},
\] 
then
\[
u_{p,q,\a,\b}\to u,\ v_{p,q,\a,\b}\to v\quad \text{ in } C^{1,\gamma}(\overline \Omega)\text{ for every $0<\gamma<1$},
\] 
where $(u,v)$ solves \eqref{eq:mainsystem} for $p=q=q_0$, $\a=\b=0$, and $E_{q_0}(u,v)=c_{\rm nod}^{q_0}$.
\end{lemma}
\begin{proof}
Take $p_n,q_n\to q_0$, $\a_n,\b_n\to 0$, and let $(u_n,v_n)$ be the corresponding least energy nodal solution of \eqref{eq:mainsystem} with $(p,q,\a,\b)=(p_n,q_n,\a_n,\b_n)$. Then $\|(u_n,v_n)\|_\infty\leq \kappa$ and, by elliptic estimates, the sequence $(u_n,v_n)$ is uniformly bounded in $W^{2,s}\times W^{2,t}$ for every $s,t>1$. Thus there exists $u,v$ such that $u_n\to u$, $v_n\to v$ in  $C^{1,\gamma}(\overline \Omega)$, and $(u,v)$ solves
\begin{equation}\label{eq:limit_prob_aux}
-\Delta u = |v|^{q_0-1}v, \qquad -\Delta v =|u|^{q_0-1}u\quad \mbox{ in } \Omega, \qquad u=v=0 \mbox{ on } \partial \Omega.
\end{equation}
Defining 
\[
w_{1n}=|x|^{\a_n}|u_n|^{p_n-1}u_n,\qquad w_{2n}=|x|^{\b_n}|v_n|^{q_n-1}v,
\]
we deduce from Lemma \ref{lemma:lower_bounds} that
\begin{align*}
&\int_\Omega |u_n^\pm|^{p_n+1} |x|^{\a_n}\, dx=\int_\Omega |w_{1n}^\pm|^\frac{p_n+1}{p_n} |x|^{-\frac{\a_n}{p_n}}\, dx\geq \eps,\\
&\int_\Omega |v_n^\pm|^{q_n+1} |x|^{\b_n}\, dx=\int_\Omega |w_{2n}^\pm|^\frac{q_n+1}{q_n} |x|^{-\frac{\b_n}{q_n}}\, dx\geq \eps,
\end{align*}
for some $\eps>0$ independent of $n$. Thus $(u,v)$ is a sign changing solution of \eqref{eq:limit_prob_aux}, and
\[
c_{\rm nod}^{q_0}\leq E_{q_0}(u,v)=\lim_n E_{p_n,q_n,\a_n,\b_n}(u_n,v_n)=\lim_n c_{\rm nod}^{p_n,q_n,\a_n,\b_n}.
\]
Combining this information with Lemma \ref{lemma:upperbound} yields the desired result.
\end{proof}

\begin{proof}[\textbf{Proof of Theorem \ref{thm:symmetrybreak_close_diag}}]
This is now an easy consequence of Lemma \ref{lemma:u=v} and Lemma \ref{lemma:continuity}. Arguing by contradiction, we would get a \emph{radial} least energy nodal solution of 
\[
-\Delta u = |u|^{q-1}u, \quad \mbox{ in } \Omega, \qquad u=0 \mbox{ on } \partial \Omega,
\]
contradicting \cite[Theorem 1.3]{AftalionPacella}.
\end{proof}

\begin{remark}
Reasoning as in this section, we can prove that the map $(p,q,\alpha,\beta)\mapsto c_{\rm nod}^{p,q,\a,\b}$ is continuous for $(p,q)$ satisfying \eqref{eq:subcritical}, $\alpha,\beta\geq 0$, and that the corresponding least energy nodal solutions converge.
\end{remark}

 \section*{Acknowledgments} 
 We would like to thank the referee for his/her valuable comments on a previous version of this manuscript, in particular for detecting a gap in the first version of the proof of Theorem 2.6.
 
 D. Bonheure is supported by INRIA - Team MEPHYSTO, 
MIS F.4508.14 (FNRS), PDR T.1110.14F (FNRS) 
\& ARC AUWB-2012-12/17-ULB1- IAPAS. 

E. Moreira dos Santos is partially supported by CNPq  projects 309291/2012-7 and 490250/2013-0 and FAPESP project 2014/03805-2. 

H. Tavares is supported by Funda\c c\~ao para a Ci\^encia e Tecnologia through the program \emph{Investigador FCT} and the project PEst-OE/EEI/LA0009/2013, and by the project ERC Advanced Grant  2013 n. 339958 ``Complex Patterns for Strongly Interacting Dynamical Systems - COMPAT''.

D. Bonheure \& E. Moreira dos Santos are partially supported by a bilateral agreement FNRS/CNPq. 
  
\medbreak

\end{document}